\documentclass[11pt,leqno]{siamltex}
\usepackage[top=1in, bottom=1in, left=1in, right=1in]{geometry}
\usepackage{amsfonts}
\usepackage{amsmath}
\usepackage{amssymb}
\usepackage{bm}
\usepackage{subfig}
\usepackage{graphicx}
\usepackage{algorithm}
\usepackage{algorithmic}
\usepackage{multirow}
\usepackage{xcolor}
\usepackage{comment}
\usepackage{tabularx}
\usepackage{booktabs}

\newcommand{\td}{\text{d}}
\newcommand{\ta}{\text{a}}
\newcommand{\Ba}{\bm{a}}
\newcommand{\Bb}{\bm{b}}
\newcommand{\Bc}{\bm{c}}

\newcommand{\Be}{\bm{e}}
\newcommand{\Bx}{\bm{x}}

\newcommand{\Bh}{\bm{h}}
\newcommand{\Bg}{\bm{g}}
\newcommand{\Bv}{\bm{v}}
\newcommand{\Bf}{\bm{f}}
\newcommand{\By}{\bm{y}}
\newcommand{\Bz}{\bm{z}}
\newcommand{\Bq}{\bm{q}}
\newcommand{\BA}{\bm{A}}
\newcommand{\BB}{\bm{B}}
\newcommand{\BI}{\bm{I}}
\newcommand{\BO}{\bm{O}}
\newcommand{\BS}{\bm{S}}
\newcommand{\BT}{\bm{T}}

\newcommand{\BW}{\bm{W}}
\newcommand{\Btheta}{\bm{\theta}}
\newcommand{\Bdelta}{\bm{\delta}}

\newcommand{\Bepsilon}{\bm{\varepsilon}}
\newcommand{\BPsi}{\bm{\Psi}}
\newcommand{\hM}{\hat{\mathcal{M}}}

\newcommand{\hf}{\hat{f}}
\newcommand{\hg}{\hat{g}}
\newcommand{\hBf}{\hat{\bm{f}}}

\newcommand{\Btau}{\bm{\tau}}

\newtheorem{remark}{Remark}[section]

\title{The Discovery of Dynamics via Linear Multistep Methods and Deep Learning: Error Estimation}
\author{Qiang Du\thanks{Department of Applied Physics and Applied Mathematics, and Data Science Institute, Columbia University, New York, NY ({\tt qd2125@columbia.edu})}
\and Yiqi Gu\thanks{Department of Mathematics, National University of Singapore, 10 Lower Kent Ridge Road, Singapore, 119076 ({\tt matguy@nus.edu.sg})}
\and Haizhao Yang\thanks{Department of Mathematics, University of Maryland, 4176 Campus Drive, College Park, MD 20742-4015, USA ({\tt hzyang@umd.edu})}
\and Chao Zhou\thanks{Department of Mathematics and Risk Management Institute, National University of Singapore, 10 Lower Kent Ridge Road, Singapore, 119076 ({\tt matzc@nus.edu.sg})}}

\begin{document}
\maketitle
\begin{abstract}
Identifying hidden dynamics from observed data is a significant and challenging task in a wide range of applications. Recently, the combination of linear multistep methods (LMMs) and deep learning has been successfully employed to discover dynamics, whereas a complete convergence analysis of this approach is still under development. In this work, we consider the deep network-based LMMs for the discovery of dynamics. We put forward error estimates for these methods using the approximation property of deep networks. It indicates, for certain families of LMMs, that the $\ell^2$ grid error is bounded by the sum of $O(h^p)$ and the network approximation error, where $h$ is the time step size and $p$ is the local truncation error order. Numerical results of several physically relevant examples are provided to demonstrate our theory.
\end{abstract}

\begin{keywords}
Discovery of Dynamics; Convergence Analysis; Data-driven Modeling; Linear Multistep Methods; Deep Learning; Lorenz System.
\end{keywords}

\begin{AMS}
65L06; 65L09; 65L20; 68T07;
\end{AMS}

\pagestyle{myheadings}
\thispagestyle{plain}
\markboth{Discovery of Dynamics via Deep Learning}{Discovery of Dynamics via Deep Learning}

\section{Introduction}

Dynamical systems are widely applied to characterize scientific principles and phenomena in various fields such as physics, biology, chemistry, economics, etc. In many situations, the observational data are accessible, whereas the underlying dynamics remain elusive. Data-driven discovery of dynamical systems is, therefore, an important research direction. There have been extensive study on data-driven discovery using Gaussian processes \cite{Kocijan2005,Raissi2017_1,Raissi2017_2,Raissi2018_2}, symbolic regression \cite{Bongard2007,Schmidt2009}, S-systems formalism \cite{Daniels2015}, sparse regression \cite{Brunton2016,Rudy2017,Zhang2018,Zhang2021}, numerical PDE analysis \cite{Kang2019}, statistical learning \cite{Lu2019}, etc. Recently, along with the rapid advancements of deep learning, the discovery of dynamics using neural networks has also been proposed  \cite{Raissi2018_1,Raissi2019,Tipireddy2019,Rudy2019,Gulgec2019,HARLIM2021109922,Qin2019,Sun2019,Long2019,Wu2020}. This paper studies high-order schemes for the discovery of dynamics using deep learning.

In numerical analysis, developing high-order methods is an important topic in many applications. Traditionally, in solving dynamical systems, high-order discretization techniques such as linear multistep methods (LMMs) and Runge-Kutta methods have been well-developed \cite{Atkinson2011,Gautschi1997,Mayers2003}. {In recent years, LMMs have also been employed for the discovery of dynamics. More precisely, one uses LMM schemes to discretize the dynamical system and take a certain class of functions to approximate the governing function. Typical candidate approximate functions include neural networks \cite{Raissi2018_3,Tipireddy2019,Xie2019}. The underlying parameters of the approximation is thereafter determined by solving the derived linear system or least squares optimization. Thanks to the high orders of the local truncation error, LMMs can discover the system with higher accuracy. Another advantage of LMMs with neural networks is the capability of approximating complicated or high-dimensional governing functions, because neural networks can overcome or lessen the curse of dimensionality for a variety of functions \cite{Guliyev2018,Montanelli2020,Montanelli2021,E2021}.
As a summary, we present an overview (Table \ref{Tab04}) of popular techniques for similar problems.}

\begin{table}[h!]
{
\centering
\begin{tabularx}{\textwidth}{p{55pt}p{130pt}>{\raggedright\arraybackslash}X>{\raggedright\arraybackslash}X}
  \toprule
   Techniques & Procedures & Main features & Limitations\\\hline
    Gaussian \newline processes \cite{Kocijan2005,Raissi2017_1,Raissi2017_2,Raissi2018_2}&
    place the Gaussian process prior on the state function; then infer parameters from data by maximizing the marginal log-likelihood &
    suitable for resolving high-dimensional problems \cite{Raissi2017_1,Raissi2017_2} &
    have restrictions on the form of the systems and are used to estimate parameters of the system \cite{Zhang2021} \\\hline

    Symbolic \newline regression \cite{Bongard2007,Schmidt2009} &
    create and correct symbolic models corresponding to the observed data &
    provide more expressive functional forms for the governing function \cite{Rudy2019} &
    computational expensive for large systems; may be prone to overfitting \cite{Brunton2016,Rudy2019} \\\hline

    Sparse \newline regression \cite{Brunton2016,Rudy2017,Zhang2018,Zhang2021} &
    find a sparse combination of candidate basis functions to approximate the governing function, whose coefficients are determined by least squares or Bayesian regression &
    provide explicit formulas of the system and do not require too much prior knowledge \cite{Zhang2021}  &
    rely on a set of appropriate candidate functions; may be inefficient for complex dynamics without a simple or sparse representation \cite{Rudy2019,Raissi2018_3,Lu2019,Long2019}\\\hline

    Statistical \newline learning \cite{Lu2019,Zhong2020} &
    learn the interaction kernel of the system in some hypothesis space by minimizing the empirical error &
    avoid the curse of dimensionality and can discover systems in very high dimensions \cite{Lu2019}&
    only work for dynamics with interaction kernel functions \cite{Feng2021}\\\hline

    LMMs with \newline neural \newline networks \cite{Raissi2018_3,Tipireddy2019,Xie2019} &
    determine the neural network approximation that minimizes the residual of the dynamical system discretized by LMMs &
    have high error orders (revealed in this paper) and can discover more complicated or high-dimensional systems \cite{Rudy2019} &
    usually provide ``black boxes", in which the mechanism of the systems is not very clearly revealed \cite{Long2019,Zhang2021}\\
   \bottomrule
\end{tabularx}
\caption{\em Overview of popular techniques for the discovery of dynamical systems.}\label{Tab04}
}
\end{table}

{Although a wide range of methods have been put forward in the past few years, theoretical analysis for those methods is less explored.} In \cite{Keller2019}, a rigorous framework based on refined notions of consistency and stability is established to yield the convergence of LMM-based discovery for three popular LMM schemes (the Adams-Bashforth, Adams-Moulton, and Backwards Differentiation Formula schemes). However, the theory in \cite{Keller2019} is specialized for methods that cannot provide a closed-form expression for the governing function, which is needed in many applications. Therefore, this paper studies the convergence theory of LMM schemes and deep learning, which can provide a closed-form description of the governing equation.

This paper concentrates on two types of discovery problems. The first type is to do the discovery on a trajectory of the dynamical system as in \cite{Raissi2018_3,Tipireddy2019,Xie2019}. In this case, the observational data are collected from a specific trajectory, and the purpose is to identify the governing function on this trajectory with a closed-form expression in the form of a neural network, the parameters of which are trained by minimizing the square residual of the corresponding LMM scheme. Through this work, we can forecast the future behavior of the same dynamics or predict the dynamics on nearby trajectories. The second type is the discovery on a compact region consisting of a bunch of trajectories on which the observational data are collected such as in \cite{Wu2019}. The purpose is to identify the governing function in a connected compact region of the domain of the governing equation, which may not have been discussed in the literature.

In this paper, we perform a convergence analysis of these methods based on the LMM framework discussed in \cite{Keller2019}. We first consider the LMMs using an abstract approximation set $\mathcal{A}$. {The main result indicates that when using a $p$-th order LMM with a step size $h$ in time, the error estimate is formally given by
\begin{equation}\label{51}
\|\hat{f}-f\|_{2,h}\leq O\left(\kappa_2(\BA_h)(h^p+e_{\mathcal{A}})\right),
\end{equation}
where $\|\cdot\|_{2,h}$ denotes the $\ell^2$ grid norm; $f$ is the true governing function and $\hat{f}$ is its approximation computed by the method; $\kappa_2(\BA_h)$ is the $2$-condition number of the corresponding matrix $\BA_h$ of the LMM; $e_{\mathcal{A}}$ is the approximation error bound between $\mathcal{A}$ and $f$ (Theorem \ref{Thm00}).} Next, based on Theorem \ref{Thm00}, we develop the error estimate of the network-based LMMs using the approximation theory of deep networks \cite{Shen2020_3,Shen2020,Lu2021,Shen2020_2}. Note that Theorem \ref{Thm00} can also be used for the error estimate of LMMs using other approximation structures. Moreover, in connection with the stability theory developed in \cite{Keller2019}, we discuss the situations that $\kappa_2(\BA_h)$ is uniformly bounded with respect to $h$. Therefore, the $\ell^2$ grid error decays to zero as $h\rightarrow 0$ and the network size approaches to infinity.

{So far, besides the mentioned work \cite{Keller2019},  some other analysis results for the discovery of dynamics can be found in, e.g., \cite{Qin2019,Wu2019}. 
In \cite{Qin2019}, the authors propose to use neural networks to approximate the flow function of the dynamical system instead of the governing function. Thereafter, they derive an error bound for the prediction of the learned model at equidistant time steps, which is conceptually given by
\begin{equation}\label{52}
|\hat{x}(kh)-x(kh)|\leq O(\|\hat{\Phi}-\Phi\|_{L^\infty}),\quad\text{with some integer}~k,
\end{equation}
where $\hat{x}$ is the predicted state of the learned dynamical system and $x$ is the true state; $\Phi$ is the true flow function and $\hat{\Phi}$ is its approximation obtained by their method. However, their analysis does not further investigate the error bound for $\|\hat{\Phi}-\Phi\|_{L^\infty}$. In \cite{Wu2019}, the authors use the linear combination of standard polynomial basis to approximate the governing function, and also estimate prediction error of the learned model at arbitrary time $t>0$, namely,
\begin{equation}\label{53}
|\hat{x}(t)-x(t)|\leq O(\|f-\text{Proj}_Vf\|_{L^\infty}+\|\hat{f}-\text{Proj}_Vf\|_{L^2}),
\end{equation}
where $V$ is the approximate polynomial space. Similarly, the two projection errors on the right-hand side are not estimated. Comparatively, we directly quantify the error between the true governing function and its approximation, which is analogous to the error terms on the right-hand sides of \eqref{52} and \eqref{53}. On one hand, our error estimate contributes to an intuitive understanding of how well the discovery is, not merely from the perspective of prediction. On the other hand, by typical ODE theory or the approach adopted in \cite{Wu2019}, the corresponding prediction error can be quickly derived based on our result.
}

In numerical experiments, {we test the performance of the network-based LMMs on both toy models and a few physically relevant benchmark problems. It is observed that for stable LMM schemes, the numerical error orders are consistent with our theory; for unstable schemes, even though the method still manages to find solutions with similar ranges of errors as some of the stable counterparts, the orders are much smaller.} We also conduct experiments to simulate the optimization errors in practice and the implicit regularization of deep learning. The results indicate that, thanks to the implicit regularization, the network-based methods without initial conditions can still find correct solutions numerically.

This paper is organized as follows. In Section 2, background knowledge of dynamical systems and LMMs is introduced. In Section 3, we will introduce the LMM approach for the discovery of dynamics and discuss auxiliary conditions for unique recovery. In Section 4, the network-based LMM approach with ReLU neural networks is described. In Section 5, we discuss the convergence rate of the preceding approach with various LMM families. Numerical experiments are provided in Section 6 to validate the theoretical results. Finally, we conclude this paper in Section 7.

\section{Dynamical Systems}
In this section, we introduce some basic notations and definitions, as adopted by \cite{Keller2019}. Most of the materials on LMMs can be found in \cite{dahlquist56,dahlquist63,henrici62}.

\subsection{Initial Value Problem}
Suppose $d>0$ is the dimension of the dynamics, let us consider the following dynamical system with an initial condition
\begin{eqnarray}
&&\frac{\td}{\td t}\Bx(t)=\Bf(\Bx(t)),\quad 0<t<T,\label{02_1}\\
&&\Bx(0)=\Bx_\text{init},\label{02_2}
\end{eqnarray}
where $\Bx\in C^\infty[0,T]^d$ is an unknown vector-valued state function; $\Bf:\mathbb{R}^d\rightarrow\mathbb{R}^d$ is a given vector-valued governing function; $\Bx_\text{init}\in\mathbb{R}^d$ is a given initial vector. To seek a numerical solution, one usually discretizes the problem by setting equidistant grid points in $[0,T]$. Let $N>0$ be an integer, we define $h:=T/N$ and $t_n=nh$ for $n=0,1,\cdots,N$. The objective for solving the initial value problem in \eqref{02_1}-\eqref{02_2} is to find an approximate value $\Bx_n\approx\Bx(t_n)$ for each $n$ when $\Bf(\Bx)$ is given.

\subsection{Linear Multistep Method}
LMMs are widely utilized in solving dynamical systems. Suppose $\Bx_0, \Bx_1, \cdots, \Bx_{M-1}$ are given states, then $\Bx_n$ for $n=M,M+1,\cdots,N$ can be computed by the following linear $M$-step scheme,
\begin{equation}\label{01}
\underset{m=0}{\overset{M}{\sum}}\alpha_m\Bx_{n-m}=h\underset{m=0}{\overset{M}{\sum}}\beta_m\Bf(\Bx_{n-m}),\quad n=M,M+1,\cdots,N,
\end{equation}
where $\alpha_m, \beta_m\in\mathbb{R}$ for $m=0,1,\cdots,M$ are specified coefficients and $\alpha_0$ is always nonzero. By the scheme, all $\Bx_n$ are evaluated iteratively from $n=M$ to $n=N$. In each step, $\Bx_{n-M},\cdots,\Bx_{n-1}$ are all given or computed previously such that $\Bx_n$ can be computed by solving algebraic equations. If $\beta_0=0$, the scheme is called explicit since $\Bx_n$ does not appear on the right hand side of \eqref{01} and $\Bx_n$ can be computed directly by
$\Bx_n=\alpha_0^{-1}\sum_{m=1}^{M}(h\beta_m\Bf(\Bx_{n-m})-\alpha_m\Bx_{n-m})$.
Otherwise, the scheme is called implicit and it requires solving nonlinear equations for $\Bx_n$. The first value $\Bx_0$ is simply set as $\Bx_0=\Bx_\text{init}$, while other initial values $\Bx_1, \cdots, \Bx_{M-1}$ need to be computed by other approaches before performing the LMM if $M>1$. Common types of LMMs include Adams-Bashforth (A-B) schemes, Adams-Moulton (A-M) schemes, and Backwards Differentiation Formula (BDF) schemes.

\subsection{Consistency}
An LMM is effective for a dynamical system only if it is consistent; that is, the discrete scheme \eqref{01} approximates the original differential equation \eqref{02_1} accurately as $h$ is small enough. More specifically, we first define the local truncation error $\Btau_{h,n}$ as
\begin{equation}\label{24}
\Btau_{h,n}=\frac{1}{h}\underset{m=0}{\overset{M}{\sum}}\alpha_m\Bx(t_{n-m})-\underset{m=0}{\overset{M}{\sum}}\beta_m\Bf(\Bx(t_{n-m})),
\end{equation}
for $n=M,M+1,\cdots,N$. Note that $\Btau_{h,n}\in\mathbb{R}^d$ is a numeric vector. It is clear that the local truncation error is defined by substituting the true function $\Bx(t)$ into the discrete scheme \eqref{01}, and measures the extent to which the true solution satisfies the discrete equation.

Now we can define the notion of consistency. The LMM \eqref{01} is said to be consistent with the differential equation \eqref{02_1} if
$\underset{M\leq n\leq N}{\max}\|\Btau_{h,n}\|_\infty\rightarrow0$ as $h\rightarrow0$,
for any $\Bx\in C^\infty[0,T]^d$. Specifically, an LMM is said to have an order $p$ if
$\underset{M\leq n\leq N}{\max}\|\Btau_{h,n}\|_\infty=O(h^p)$ as $h\rightarrow0$.

\section{Discovery of Dynamics}\label{Sec_LMM_discovery}
In this section, we introduce the discovery of dynamics on a single trajectory, on which a time series of the state is available. Conventional LMMs with auxiliary conditions for this type of discovery are introduced. Note that these methods can be simply generalized for the discovery on a compact region, which will be discussed in Section \ref{Sec_subsets}.

The discovery of dynamics is essentially an inverse process of solving a dynamical system \eqref{02_1}-\eqref{02_2} with given observations on the state. That is, suppose that only the information of the state $\Bx$ at the equidistant time steps $\{t_n\}_{n=0}^N$ are provided, we would like to recover $\Bf$, namely, the governing function of the state.

\subsection{Linear Multistep Method}
Let $\Bx(t)\in C^\infty([0,T])^d$ and $\Bf(\cdot):\mathbb{R}^d\rightarrow\mathbb{R}^d$ be two vector-valued functions satisfying the dynamics \eqref{02_1}, and we assume $\Bx(t)$ and $\Bf(\cdot)$ are both unknown. Now given $\Bx_n=\Bx(t_n)$ for $n=0,\cdots,N$, the objective is to determine $\Bf(\cdot)$, i.e. to find a closed-form expression for $\Bf(\cdot)$ or to evaluate $\Bf(\Bx_n)$ for all $n$. One effective approach is to build a discrete relation between $\Bx_n$ and $\Bf_n\approx\Bf(\Bx_n)$ by LMMs \cite{Keller2019}, namely,
\begin{equation}\label{04}
h\underset{m=0}{\overset{M}{\sum}}\beta_m\Bf_{n-m}=\underset{m=0}{\overset{M}{\sum}}\alpha_m\Bx_{n-m},\quad n=M,M+1,\cdots,N,
\end{equation}
where $\Bf_n\in\mathbb{R}^d$ is an approximation of $\Bf(\Bx_n)$. Note that \eqref{04} directly follows the LMM scheme \eqref{01}. Different from \eqref{01} that evaluates $\Bx_n$ given $\Bf$, \eqref{04} computes $\Bf_n$ from the data $\Bx_n$. It indicates the dynamics discovery is actually an inverse process of solving the dynamical system \cite{Keller2019}.
Moreover, we note that the components of $\Bf(\cdot)$
can be discovered independently. Thus, in the remainder of this paper, without loss of generality, we work with a scalar-valued system to simplify \eqref{04}
using notation in a scalar form as the following general equation,
\begin{equation}\label{09}
h\underset{m=0}{\overset{M}{\sum}}\beta_mf_{n-m}=\underset{m=0}{\overset{M}{\sum}}\alpha_mx_{n-m},\quad n=M,M+1,\cdots,N.
\end{equation}

It is worth noting that $f_n$ may not be involved in \eqref{09} for some indices $n$ between $0$ and $N$. For example, in A-B schemes, $f_N$ does not appear in \eqref{09} since $\beta_0=0$. In general, given an LMM, we use $s$ and $e(N)$ to denote the first and last indices such that $f_s$ and $f_{e(N)}$ are involved in \eqref{09} with non-zero coefficients (correspondingly, $\beta_{M-s}$ and $\beta_{N-e(N)}$ are both nonzero). We also write $t(N):=e(N)-s+1$ as the total number of $f_n$ involved in \eqref{09}. We briefly list $s$, $e(N)$, $t(N)$, and the truncation error orders $p$ of A-B, A-M, and BDF schemes in Table \ref{Tab_LMM}.

\begin{table}
\centering
\begin{tabular}{|c|c|c|c|c|}
  \hline
  Method & $s$ & $e(N)$ & $t(N)$ & $p$ \\\hline
  $M$-step A-B & $0$ & $N-1$ & $N$ & $M$\\\hline
  $M$-step A-M & $0$ & $N$ & $N+1$ & $M+1$\\\hline
  $M$-step BDF & $M$ & $N$ & $N-M+1$ & $M$\\\hline
\end{tabular}
\caption{\em The first involved index $s$, the last involved index $e(N)$, the total number of involved indices $t(N)$, and the truncation error order $p$ for common types of LMMs.}
\label{Tab_LMM}
\end{table}

\subsection{Auxiliary Conditions}
For each linear $M$-step method, it is supposed to compute all unknowns $\{f_n\}_{n=s}^{e(N)}$ by the linear relation \eqref{09}. In the following, we will use the special notation $\vec{\cdot}$ and bold fonts to denote column vectors of size $O(N)$, distinguishing them from other vectors or vector functions. We write
\begin{gather}
\vec{\Bf}_h:=\left[f_s\quad f_{s+1}\quad\cdots\quad f_{e(N)}\right]^T\in\mathbb{R}^{t(N)},\\
\vec{\Bq}_h:=\frac{1}{h}\left[\underset{m=0}{\overset{M}{\sum}}\alpha_mx_{M-m}\quad\underset{m=0}{\overset{M}{\sum}}\alpha_mx_{M+1-m}\quad\cdots\quad\underset{m=0}{\overset{M}{\sum}}\alpha_mx_{N-m}\right]^T\in\mathbb{R}^{N-M+1},
\end{gather}
\begin{multline}\label{16}
\BB_h:=\left[\begin{array}{cccccccc}
              \beta_{M-s} & \beta_{M-s-1} &\cdots & \beta_{N-e(N)} &  &  &  \\
               & \beta_{M-s} & \beta_{M-s-1} & \cdots & \beta_{N-e(N)} &  &  \\
               &  & \ddots & \ddots & \ddots & \ddots &  \\
               &  &  & \beta_{M-s} & \beta_{M-s-1} & \cdots & \beta_{N-e(N)}
            \end{array}
\right]\\\in\mathbb{R}^{(N-M+1)\times t(N)},
\end{multline}
then \eqref{09} leads to the following linear system,
\begin{equation}\label{22}
\BB_h\vec{\Bf}_h=\vec{\Bq}_h.
\end{equation}

However, the number of equations and unknowns may not be equal in \eqref{09}. For A-B and A-M schemes, it is insufficient to determine $\{f_n\}_{n=s}^{e(N)}$ by \eqref{09} since equations are fewer than unknowns. This implies that the linear system \eqref{22} is underdetermined. For this issue, a natural solution is to provide $N_a:=t(N)-(N-M+1)$ auxiliary linear conditions to make $\{f_n\}_{n=s}^{e(N)}$ unique. For example, we can compute $N_a$ certain unknown $f_n$ directly by first-order (derivative) finite difference method (FDM) using related data. For consistency, the selected FDM should be of the same error order as the LMM. Assume the LMM has order $p$, one straightforward way is to compute the initial $N_a$ unknowns by one-sided FDM of order $p$, i.e.,
\begin{equation}\label{20}
f_n=\frac{1}{h}\underset{m=0}{\overset{p}{\sum}}\gamma_mx_{n+m},\quad n=s,s+1,\cdots,s+N_a-1,
\end{equation}
where $\gamma_m$ are the corresponding finite difference coefficients. Note that \eqref{20} has the error estimate
\begin{equation}\label{25}
\underset{s\leq n\leq s+N_a-1}{\max}|f_n-f(\Bx(t_n))|=O(h^p),\quad\text{as}~h\rightarrow0.
\end{equation}

If we write
\begin{equation}\label{eqn:ch}
\Bc_h:=\frac{1}{h}\left[\underset{m=0}{\overset{p}{\sum}}\gamma_mx_{s+m}\quad\underset{m=0}{\overset{p}{\sum}}\gamma_mx_{s+1+m}\quad\cdots\quad\underset{m=0}{\overset{p}{\sum}}\gamma_mx_{s+N_a-1+m}\right]^T\in\mathbb{R}^{N_a},
\end{equation}
then combining \eqref{09} and \eqref{20} leads to the following augmented linear system
\begin{equation}\label{21}
\BA_h\vec{\Bf}_h=\left[\begin{array}{c}\Bc_h\\\vec{\Bq}_h\end{array}\right],
\end{equation}
where
\begin{equation}\label{28}
\BA_h:=\left[\begin{array}{c}
              \bm{C}
               \\
              \BB_h
            \end{array}
\right]
\quad\text{and}\quad
\bm{C}:=\left[\begin{array}{c}
              \begin{array}{cc}
                \BI_{N_a} & \BO
              \end{array}
            \end{array}
\right]
\end{equation}
with $\BI_{N_a}$ being the $N_a\times N_a$ identity matrix and $\BO$ being the zero matrix of size $N_a\times \left(t(N)-N_a\right)$. Clearly, \eqref{21} has a unique solution since the coefficient matrix is lower triangular with nonzero diagonals. Moreover, if $M\ll N$, the linear system \eqref{21} is sparse.

In general, as pointed out in \cite{Keller2019}, we can formulate the auxiliary conditions in various ways, not just as discussed above. Different auxiliary conditions, such as initial and terminal conditions have different effects on the stability and the convergence of the method, see further discussions in \cite{Keller2019}. An interesting question is whether the regularization effect provided by the neural network approximations could help mitigate these effects.

\section{Neural Network Approximation}
In this section, we first introduce the concept of fully connected neural networks (FNNs) and their approximation properties. Next, the network-based LMMs for the discovery on a trajectory will be presented together with a discussion on implicit regularization. Finally, we discuss the discovery on a compact region.

\subsection{Preliminaries}
We introduce the fully connected neural network (FNN) which is widely used in deep learning. Mathematically speaking, given an activation function $\sigma$, $L\in\mathbb{N}^+$,  and $W_\ell\in\mathbb{N}^+$ for $\ell=1,\dots,L$, an FNN is the composition of $L$ simple nonlinear functions, called hidden layer functions, in the following formulation:
\begin{equation}
\hat{\phi}(\Bx;\Btheta):=\Ba^T \Bh_{L} \circ \Bh_{L-1} \circ \cdots \circ \Bh_{1}(\Bx)\quad \text{for } \Bx\in\mathbb{R}^d,
\end{equation}
where $\Ba\in \mathbb{R}^{W_L}$; $\bm{h}_{\ell}(\Bx_{\ell}):=\sigma\left(\BW_\ell \Bx_{\ell} + \Bb_\ell \right)$ with $\BW_\ell \in \mathbb{R}^{W_{\ell}\times W_{\ell-1}}$ and $\Bb_\ell \in \mathbb{R}^{W_\ell}$ for $\ell=1,\dots,L$. With the abuse of notations, $\sigma(\Bx)$ means that $\sigma$ is applied entry-wise to a vector $\Bx$ to obtain another vector of the same size. $W_\ell$ is the width of the $\ell$-th layer and $L$ is the depth of the FNN. $\Btheta:=\{\Ba,\,\BW_\ell,\,\Bb_\ell:1\leq \ell\leq L\}$ is the set of all parameters in $\hat{\phi}$ to determine the underlying neural network. Common types of activation functions include the rectified linear unit (ReLU) $\max(0,x)$ and the sigmoid function $(1+e^{-x})^{-1}$.

\subsection{Approximation Property}
Now let us introduce existing results on the approximation property of ReLU FNNs. Given a function $g$ on a compact subset $\mathcal{S}$ in $\mathbb{R}^d$, we can define the modulus of continuity by
\begin{equation}\label{14}
\omega_g(\lambda)=\sup\{|g(\Bx)-g(\By)|:\|\Bx-\By\|_2\leq \lambda,~\Bx,\By\in\mathcal{S}\},\quad\text{for~any}~\lambda\geq0,
\end{equation}
where $\|\Bx\|_2:=\sqrt{x_1^2+x_2^2+\cdots+x_d^2}$ is the Euclidean norm of a vector in $\mathbb{R}^d$. Suppose $\Lambda$ is any subset in $\mathbb{R}^d$, we define the $C^r$ norm in $\Lambda$,
\begin{equation}\label{30}
\|g\|_{C^r(\Lambda)}:=\max\left\{\|\partial^{\bm \alpha} g\|_{L^\infty(\Lambda)}:\|{\bm \alpha}\|_1\leq r,{\bm \alpha}\in\mathbb{N}^d\right\}.
\end{equation}
Besides, we define
\begin{equation}\label{12}
R_\Lambda:=\inf\{\rho>0:\Lambda\subset[-\rho,\rho]^d\}
\end{equation}
as the ``radius" of $\Lambda$.

Approximation properties of ReLU FNNs for continuous functions and smooth functions are indicated as follows.
\begin{proposition}\label{Pro01}
Given any $J,K\in\mathbb{N^+}$ and a function $g$ on a compact subset $\mathcal{S}$ of $\mathbb{R}^d$,
\begin{enumerate}
  \item if $g\in C(\mathcal{S})$, there exists a ReLU FNN $\hat{\phi}$ with width $3^{d+3}\max\{d\lfloor J^{1/d}\rfloor, J+1\}$ and depth $12K+2d+14$ such that
  \begin{equation}\label{33_1}
|\hat{\phi}(\Bx)-g(\Bx)|\leq 19\sqrt{d}\omega_g(2R_\mathcal{S}J^{-2/d}K^{-2/d}),\quad\text{for~any}~\Bx\in\mathcal{S};
\end{equation}
  \item if $g\in C^r(\mathcal{S})$ with $r\in\mathbb{N^+}$, there exists a ReLU FNN $\hat{\phi}$ with width $17r^{d+1}3^dd(J+2)\log_2(8J)$ and depth $18r^2(K+2)\log_2(4K)+2d$ such that
  \begin{equation}\label{33_2}
|\hat{\phi}(\Bx)-g(\Bx)|\leq 170R_\mathcal{S}(r+1)^d8^r\|g\|_{C^r(\mathcal{S})}J^{-2r/d}K^{-2r/d},\quad\text{for~any}~\Bx\in\mathcal{S},
\end{equation}
\end{enumerate}
\end{proposition}

The estimate \eqref{33_1} directly follows Theorem 4.3 in \cite{Shen2020_3}, and the estimate \eqref{33_2} can be derived from Theorem 1.1 in \cite{Lu2021} by generalizing the regular domain $[0,1]^d$ to a compact subset $\mathcal{S}$.

Note that the error bounds in \eqref{33_1} and \eqref{33_2} suffer from the curse of dimensionality; namely, they exponentially depend on the dimension of the whole space $\mathbb{R}^d$. {However, if we are only interested in the approximation on a low-dimensional submanifold rather than a general compact subset in $\mathbb{R}^d$, stronger results can be adopted. Specifically, we consider a submanifold having certain volume, condition number, and geodesic covering regularity. Note that for manifolds, the definition of volume can be found in \cite{Maeda1978,Boya2003}, and the condition number and geodesic covering regularity are formally defined by Definition 2.1-2.3 in \cite{Baraniuk2009}.} The approximation properties on submanifolds are given as follows.

\begin{proposition}\label{Pro02}
Given $J,K\in\mathbb{N^+}$, $\varepsilon\in(0,1)$, $\delta\in(0,1)$. Let $\mathcal{M}\subset\mathbb{R}^d$ be a compact $d_\mathcal{M}$-dimensional Riemannian submanifold having condition number $\tau_\mathcal{M}^{-1}$, volume $V_\mathcal{M}$, and geodesic covering regularity $G_\mathcal{M}$, and define the $\varepsilon$-neighborhood as
$\mathcal{M}_\varepsilon:=\{\Bx\in\mathbb{R}^d:\underset{\By\in\mathcal{M}}{\inf}\|\Bx-\By\|_2\leq\varepsilon\}$. Suppose $g$ is a function defined in $\mathcal{M}_\varepsilon$,
\begin{enumerate}
  \item if $g\in C(\mathcal{M}_\varepsilon)$, there exists a ReLU FNN $\hat{\phi}$ with width $3^{d_\delta+3}\max\{d_\delta\lfloor J^{1/d_\delta}\rfloor, J+1\}$ and depth $12K+2d_\delta+14$ such that
  \begin{multline}\label{34_1}
|\hat{\phi}(\Bx)-g(\Bx)|\leq 2\omega_f\left(4R_\mathcal{M}\varepsilon((1-\delta)^{-1}\sqrt{d/d_\delta}+1)\right)\\
+19\sqrt{d}\omega_g\left(4R_\mathcal{M}(1-\delta)^{-1}\sqrt{d/d_\delta}J^{-2/d_\delta}K^{-2/d_\delta}\right),\quad\text{for~any}~\Bx\in\mathcal{M}_\varepsilon;
\end{multline}
  \item if $g\in C^r(\mathcal{M}_\varepsilon)$ with $r\in\mathbb{N^+}$, there exists a ReLU FNN $\hat{\phi}$ with width $17r^{d_\delta+1}3^{d_\delta} d_\delta(J+2)\log_2(8J)$ and depth $18r^2(K+2)\log_2(4K)+2d_\delta$ such that
  \begin{multline}\label{34_2}
|\hat{\phi}(\Bx)-g(\Bx)|\leq 8\|g\|_{C^r(\mathcal{M}_\varepsilon)}R_\mathcal{M}\varepsilon((1-\delta)^{-1}\sqrt{d/d_\delta}+1)\\ +170R_\mathcal{M}(r+1)^{d_\delta}8^r(1-\delta)^{-1}\|g\|_{C^r(\mathcal{M}_\varepsilon)}J^{-2r/d_\delta}K^{-2r/d_\delta},\quad\text{for~any}~\Bx\in\mathcal{M}_\varepsilon,
\end{multline}
\end{enumerate}
where $d_\delta:=O\left(d_\mathcal{M}\ln\left(dV_\mathcal{M}G_\mathcal{M}\tau_\mathcal{M}^{-1}/\delta\right)/\delta^2\right)=O\left(d_\mathcal{M}\ln(d/\delta)/\delta^2\right)$ is an integer with $d_\mathcal{M}\leq d_\delta\leq d$.
\end{proposition}

Equation \eqref{34_1} in Proposition \ref{Pro02} is an immediate result of Theorem 1.2 in \cite{Shen2020_3} and Equation \eqref{34_2} can be derived from Theorem 1.1 in \cite{Lu2021} and Theorem 4.4 in \cite{Shen2020_3} similarly. In Proposition \ref{Pro02}, both the error bounds and the ReLU FNN sizes depend on $d_\delta$ instead of $d$ so that the curse of dimensionality is lessened. Note that when $\delta$ is closer to 1, $d_\delta$ is closer to $d_\mathcal{M}$, then the approximation actually occurs in a reduced space with dimension close to $d_\mathcal{M}$ instead of the whole space $\mathbb{R}^d$.

The approximation properties of other FNNs are also studied. For example, the properties of the Floor-ReLU FNN and a special three-hidden-layer FNN can be found in \cite{Shen2020} and \cite{Shen2020_2}, respectively. Also, dimension-independent error bounds of FNNs for the target functions in Barron space are investigated in \cite{Barron1992}. It is also interesting to apply these approximation theories to develop error estimates of dynamics discovery as future work.

\subsection{Network-based Methods for Discovery}\label{Sec_one_trajectory}
Let us review the discovery of dynamics on a single trajectory introduced in Section \ref{Sec_LMM_discovery}. Indeed, the discovery by conventional LMMs is simple to implement, and the solution can be found by merely solving a linear system. However, the governing function $\Bf$ is only computed at prescribed equidistant time steps, and the relation between $\Bf$ and the state $\Bx$ is still unknown. One strategy to overcome this limitation is to approximate each component of $\Bf$ by functions of specific structures such as neural networks, polynomials, splines, etc. The approximate functions can be determined through optimization and will serve as closed-form expressions for $\Bf$. In real applications, once $\Bf$ has been recovered with an explicit expression, the future behavior of the $\Bx$ on the same trajectory can be forecast via solving \eqref{02_1}-\eqref{02_2} with the given initial condition. On the other hand, the behavior of the $\Bx$ on nearby trajectories can also be predicted via solving \eqref{02_1}-\eqref{02_2} with perturbed initial conditions.

Among all structures of approximations, it is popular to employ neural networks in the discovery problems. Especially, when $d$ is moderately large, it is convenient to use neural networks to approximate the governing functions with high-dimensional inputs, which is usually intractable for other structures. Therefore we focus on the network-based methods in this paper. Note that the proposed methods can be easily generalized for other structures of approximations.

We consider the neural network approximation based on the LMM scheme \eqref{09}. Generally, we use $\mathcal{N}_{\hM}$ to denote the set of all neural networks with a specified architecture of a size set $\hM$. For example, $\mathcal{N}_{\hM}$ can be the set of all FNNs with the fixed size $\hM=\{L,W\}$, where $L$ is the depth and $W$ is the width. The notation $\hM\rightarrow\infty$ means that some of the numbers in $\hM$ go to infinity.

Now we introduce a network $\hf_{\hM}(\Bz)\in\mathcal{N}_{\hM}$ to approximate $f(\Bz)$, an arbitrary component of $\Bf(\cdot)$. The neural network method can be developed by replacing $f_{n}$ with $\hf_{\hM}(\Bx_n)$ in \eqref{09}, namely,
\begin{equation}\label{05}
h\underset{m=0}{\overset{M}{\sum}}\beta_m\hf_{\hM}(\Bx_{n-m})=\underset{m=0}{\overset{M}{\sum}}\alpha_mx_{n-m},\quad n=M,M+1,\cdots,N,
\end{equation}
where $\Bx_n$ for $n=0,\cdots,N$ are given sample locations.

Unfortunately, if $\hM$ is too small, the degree of freedom of $\mathcal{N}_{\hM}$ will be less than the number of equations in \eqref{05} and, hence, there is no $\hf_{\hM}\in\mathcal{N}_{\hM}$ such that \eqref{05} is satisfied precisely. Even if $\hM$ is large enough, it is usually intractable to solve \eqref{05} for $\hf_{\hM}$ directly because of the nonlinear parametrization of neural networks. Consequently, in practice, we seek $\hf_{\hM}$ by minimizing the residual of \eqref{05} under a machine learning framework. Namely, we aim to find $\hf_{\hM}\in\mathcal{N}_{\hM}$ such that
\begin{equation}\label{08_1}
J_h(\hf_{\hM})=\underset{u\in\mathcal{N}_{\hM}}{\min}J_h(u),
\end{equation}
where
\begin{equation}\label{08_2}
J_h(u):=\frac{1}{N-M+1}\underset{n=M}{\overset{N}{\sum}}\left|\underset{m=0}{\overset{M}{\sum}}\beta_mu(\Bx_{n-m})-\underset{m=0}{\overset{M}{\sum}}h^{-1}\alpha_mx_{n-m}\right|^2.
\end{equation}

However, similar to the underdetermined linear system \eqref{22} that has infinitely many solutions, there exist infinitely many sets of real numbers $\{y_n\}_{n=s}^{e(N)}$ such that $J_h(u)=0$ providing
\begin{equation}\label{06}
u(\Bx_n)=y_n,\quad\forall n.
\end{equation}
For each set $\{y_n\}_{n=s}^{e(N)}$, if the degree of freedom of $\mathcal{N}_{\hM}$ is large enough, there is always some $u\in\mathcal{N}_{\hM}$ such that \eqref{06} is satisfied due to overfitting. In this situation, $u$ is a global minimizer of $J_h$. Consequently, $J_h$ admits infinitely many global minimizers, all of which lead to $J_h=0$ but take distinct values at $\{\Bx_n\}_{n=s}^{e(N)}$. It implies a minimizer of $J_h$ might be totally different from the target governing function we aim to approximate.

To ensure the uniqueness of the minimizer in the function space at grid points, we introduce auxiliary conditions and build an augmented loss function based on \eqref{08_2}. For example, the initial condition \eqref{20} on the solution network $\hf_{\hM}$ is enforced by solving
\begin{equation}
J_{\ta,h}(\hf_{\hM})=\underset{u\in\mathcal{N}_{\hM}}{\min}J_{\ta,h}(u),\label{23_1}
\end{equation}
where
\begin{equation}
J_{\ta,h}(u):=\frac{1}{t(N)}\left(\underset{n=s}{\overset{s+N_a-1}{\sum}}\left|u(\Bx_n)-\frac{1}{h}\underset{m=0}{\overset{p}{\sum}}\gamma_mx_{n+m}\right|^2+\underset{n=M}{\overset{N}{\sum}}\left|\underset{m=0}{\overset{M}{\sum}}\beta_mu(\Bx_{n-m})-\underset{m=0}{\overset{M}{\sum}}h^{-1}\alpha_mx_{n-m}\right|^2\right).\label{23_2}
\end{equation}
The augmented optimization above guarantees that $\hf_{\hM}(\Bx_n)= \hg_{\hM}(\Bx_n)$ for $n=s,\cdots,e(N)$ providing $J_{\ta,h}(\hf_{\hM})= J_{\ta,h}(\hg_{\hM})= 0$, for any $\hf_{\hM},\hg_{\hM}\in\mathcal{N}_{\hM}$.

Indeed, two networks that are equal at grids $\{\Bx_n\}_{n=s}^{e(N)}$ are not necessarily equal on the whole trajectory $\{\Bx(t):0\leq t\leq T\}$. Fortunately, it is shown for regression problems and partial differential equation problems, deep learning can generalize well \cite{Kawaguchi2017,Mei2018,Mei2019,Luo2020}. This means the closeness of two networks at a dense set of training inputs can lead to their closeness at other nearby inputs. It can be inferred that $\hf_{\hM}(\Bx(t))\approx\hg_{\hM}(\Bx(t))$ for $0\leq t\leq T$ providing $\hf_{\hM}(\Bx_n)= \hg_{\hM}(\Bx_n)$ for $n=s,\cdots,e(N)$ for any $\hf_{\hM},\hg_{\hM}\in\mathcal{N}_{\hM}$ as long as $N$ is moderately large.

\subsection{Implicit Regularization}\label{Sec_implicit_regularization}
We discuss the implicit regularization \cite{Neyshabur2017,Lei2018} of gradient descent in deep learning. For regression problems, if we use over-parameterized FNNs with the standard random initialization, gradient descent can lead to global convergence with a linear convergence rate under certain conditions \cite{Jacot2018,Chen2019,Du2018,Zhu2019}. Similar results also exist in the problems of solving partial differential equations \cite{Luo2020}. Even though the global convergence could be established with over-parametrization, global minimizers are typically not unique. It is interesting to investigate what global minimizers would be identified by gradient descent and how the training process would reduce fitting errors. To answer these questions, it has been shown that, in regression problems, the training of FNN first captures low-frequency components of the target function and then starts to eliminate the high-frequency fitting error \cite{Xu2019,Luo2019}. Similar work about this spectral bias of deep learning is discussed in \cite{Cao2019,pmlr-v97-rahaman19a}. In sum, all the above discussions show that neural networks trained by gradient descent in regression problems have an implicit bias towards smooth functions with low frequencies among all possible neural networks that perfectly fit training data.

Now let us consider the preceding network-based LMM optimization. Note that the loss function \eqref{08_2} without auxiliary conditions and the loss function \eqref{23_2} with auxiliary conditions are formally close to the $\ell^2$ loss in regression problems. Especially, for BDF schemes, $\beta_0=1$ and $\beta_2=\beta_3=\cdots=0$, so the loss functions \eqref{08_2} and \eqref{23_2} are exactly the $\ell^2$ loss. Hence, it is conjectured that the implicit regularization discussed above can also be applied to the LMM optimization. Namely, the gradient descent tends to find a very smooth function among all global minimizers. Consequently, if the target governing function is also smooth enough, the gradient descent is expected to find good approximations either through \eqref{08_1} without auxiliary conditions, or through \eqref{23_1} with auxiliary conditions. Numerical experiments in Section \ref{Sec_numerical} will validate this.

{
However, the implicit regularization may not succeed in the discovery problems with noisy measurement. In a recent work \cite{Xie2019}, a typical example is presented to show the discovery of the Navier-Stokes equation using A-M scheme with $M=1$, where the data is perturbed with Gaussian noise. Similar to the approach discussed in this work, the network is trained through the optimization with implicit regularization. The results show that the discovery is fairly accurate (with errors $O(10^{-2})$) for small noise magnitude ($1\%$), but becomes completely incorrect (with errors $O(10^1)$) if the noise is enlarged to $5\%$. This implies that the network approximation with implicit regularization is sometimes sensitive to the perturbation of the raw data such as noise, especially when the problem is ill-conditioned. Future investigations should be carried out to make further potential improvement for this issue.
}

\subsection{Discovery on a Compact Region}\label{Sec_subsets}
The network-based formulation \eqref{23_1}-\eqref{23_2} is specific for the discovery on a single trajectory from which the data are collected. More generally, we can build similar formulations for the discovery on a connected compact region, from which a set of trajectories can be sampled, to recover the whole vector field in this region.

Suppose $\Bx(t;\tilde{\Bx}_0)$ is the solution of \eqref{02_1} with initial value $\tilde{\Bx}_0$. Let $\Gamma$ be a compact subset in $\mathbb{R}^d$, then $\Omega:=\{\Bx(t;\tilde{\Bx}):0\leq t\leq T,\tilde{\Bx}\in\Gamma\}\subset\mathbb{R}^d$ is a compact region filled with all trajectories starting from $\Gamma$ with time period $0\leq t\leq T$. In practice, suppose we are given a dataset $\{\Bx_{n,n'}=\Bx(t_n;\tilde{\Bx}_{n'})\}_{n=0,\cdots,N;n'=1,\cdots,N'}$, where $\{\tilde{\Bx}_{n'}\}_{n'=1,\cdots,N'}$ is a set of points densely distributed in $\Gamma$, and suppose $\Omega$ is densely covered by $\{\Bx_{n,n'}\}$. We aim to use neural networks to approximate the governing function in the whole subset $\Omega$.

Note that \eqref{23_2} is a loss function with respect to one trajectory. For multiple trajectories, we can build a similar loss function by summing up all individual loss functions with respect to each trajectory. Specifically, let $\hf_{\hM}$ be a network that approximates a certain component of the governing function, then we can determine $\hf_{\hM}$ by
\begin{equation}
J_{\ta,h,\text{multi}}(\hf_{\hM})=\underset{u\in\mathcal{N}_{\hM}}{\min}J_{\ta,h,\text{multi}}(u),\label{29_1}
\end{equation}
where
\begin{multline}
J_{\ta,h,\text{multi}}(u):=\frac{1}{N't(N)}\underset{n'=1}{\overset{N'}{\sum}}\\
\cdot\left(\underset{n=s}{\overset{s+N_a-1}{\sum}}\left|u(\Bx_{n,n'})-\frac{1}{h}\underset{m=0}{\overset{p}{\sum}}\gamma_mx_{i+m,n'}\right|^2+\underset{n=M}{\overset{N}{\sum}}\left|\underset{m=0}{\overset{M}{\sum}}\beta_mu(\Bx_{n-m,n'})-\underset{m=0}{\overset{M}{\sum}}h^{-1}\alpha_mx_{n-m,n'}\right|^2\right).\label{29_2}
\end{multline}

Similar to the discovery on a single trajectory, the optimization \eqref{29_1}-\eqref{29_2} for multiple trajectories will be also effective without auxiliary conditions due to the implicit regularization.

\section{Convergence Analysis}\label{Sec_theory}
In this section, we consider the convergence of the preceding network-based dynamics discovery using LMMs, namely, the convergence from the global minimizer of the optimization to the exact governing function $f$ as $\hM\rightarrow\infty$ and $h\rightarrow0$. The optimization with auxiliary initial conditions is taken as a special case for analysis. For the optimization with other auxiliary conditions, a similar argument can be applied.

\subsection{Error Estimates on a Trajectory}
We consider the error estimation of the discovery on the specific trajectory $\mathcal{T}:=\{\Bx(t):0\leq t\leq T\}$.
For least-square optimization, people are usually interested in the $\ell^2$-type error estimation. Therefore, let us introduce the $\ell^2$ seminorm $|g|_{2,h}:=\left((N+1)^{-1}\sum_{n=0}^{N}|g(\Bx_n)|^2\right)^{1/2}$, for all $g\in C(\mathcal{T})$ with a given $h>0$. Note that $|\cdot|_{2,h}$ is not a norm in $C(\mathcal{T})$ since $|g|_{2,h}=0$ does not imply $g=0$ in $C(\mathcal{T})$. However, $|\cdot|_{2,h}$ acts as a norm in the space of all grid functions merely defined on $\{\Bx_n\}_{n=0}^N$ (see \cite{Keller2019}).

As discussed above, for a specific LMM, some states in $\{\Bx_n\}_{n=0}^N$ may not be involved in the scheme. For fairness, we study the convergence at all involved states $\{\Bx_n\}_{i=s}^{e(N)}$. Therefore, we rewrite $|\cdot|_{2,h}$ as the LMM-related seminorm $|g|_{2,h}=\left(t(N)^{-1}\sum_{n=s}^{e(N)}|g(\Bx_n)|^2\right)^{1/2}$, for all $g\in C(\mathcal{T})$.

Without ambiguity, we use the notation $|\cdot|_{2,h}$ for all LMMs afterwards. If we write $\{g(\Bx_n)\}_{n=s}^{e(N)}$ as a vector
$\vec{\Bg}:=\left[g(\Bx_s)\quad g(\Bx_{s+1})\quad\cdots\quad g(\Bx_{e(N)})\right]^T$, then it follows $|g|_{2,h}=(t(N))^{-1/2}\|\vec{\Bg}\|_2$,
where $\|\cdot\|_2$ is the Euclidean norm of a column vector.

First, let us reformulate the optimization \eqref{23_1}-\eqref{23_2} with an abstract admissible set, say, $J_{\ta,h}(\hf_{\mathcal{A},h})=\underset{u\in\mathcal{A}}{\min}J_{\ta,h}(u)$, where $J_{\ta,h}(u)$ is defined in \eqref{23_2} and $\mathcal{A}$ is a general nonempty set of functions. We aim to estimate the distance between $\hf_{\mathcal{A},h}$ and $f$.

For a given LMM, recall that $\BB_h$ defined in \eqref{16} is constructed by lining up the LMM coefficients into rows and $\BA_h$ is defined in \eqref{21}. We denote the 2-condition number of $\BA_h$ by $\kappa_2(\BA_h)=\|\BA_h\|_2\|\BA_h^{-1}\|_2$. The estimation is described as follows.

\begin{theorem}\label{Thm00}
In the dynamical system \eqref{02_1}, suppose $\Bx\in C^\infty([0,T])^d$ and $\Bf$ is defined in $\mathcal{T}'$, a small neighborhood of $\mathcal{T}$. Let $f$ be an arbitrary component of $\Bf$. Also, let $N>0$ be an integer and $h:=T/N$, then we have
\begin{equation}
\left|\hf_{\mathcal{A},h}-f\right|_{2,h}<C\kappa_2(\BA_h)\left(h^p+e_\mathcal{A}\right),
\end{equation}
where $\hf_{\mathcal{A},h}\in\mathcal{A}$ is a global minimizer of $J_{\ta,h}$ defined by \eqref{23_2} corresponding to an LMM with order $p$; $e_\mathcal{A}$ satisfies $e_\mathcal{A}>\underset{u\in\mathcal{A}}{\inf}\underset{\Bx\in\mathcal{T}'}{\sup}|u(\Bx)-f(\Bx)|$; $C$ is a constant independent of $h$ and $\mathcal{A}$.
\end{theorem}

\begin{proof}
Given $h>0$, similar to \eqref{24}, we can define the component-wise local truncation error by
$\tau_{h,n}:=h^{-1}\sum_{m=0}^{M}\alpha_m x_{n-m}-\sum_{m=0}^{M}\beta_m f(\Bx(t_{n-m}))$.
Then by denoting
\begin{equation}\label{41}
\vec{\Btau}_h:=\left[\tau_{h,M}\quad \tau_{h,M+1}\quad\cdots\quad\tau_{h,N}\right]^T,\quad\vec{\Bf}:=\left[f(\Bx_s)\quad f(\Bx_{s+1})\quad\cdots\quad f(\Bx_{e(N)})\right]^T,
\end{equation}
we have $\vec{\Btau}_h = \vec{\Bq}_h-\BB_h\vec{\Bf}$. By the hypothesis that the LMM has order $p$, there exists some $C_1>0$ independent of $h$ such that
\begin{equation}\label{26}
\|\vec{\Btau}_h\|_2\leq(N-M+1)^\frac{1}{2}\|\vec{\Btau}_h\|_\infty < C_1(N-M+1)^\frac{1}{2}h^p.
\end{equation}

On the other hand, since $e_\mathcal{A}>\underset{u\in\mathcal{A}}{\inf}\underset{\Bx\in\mathcal{T}'}{\sup}|u(\Bx)-f(\Bx)|$, there exists a function $v\in\mathcal{A}$ such that
\begin{equation}\label{10}
|v(\Bx)-f(\Bx)|\leq e_\mathcal{A},\quad\forall\Bx\in\mathcal{T}'.
\end{equation}

Also, write
$\Be_h=\Bc_h-\left[\BI_{N_a}~\BO\right]\vec{\Bf}$, where $\Bc_h$ is defined in \eqref{eqn:ch}. Then by \eqref{25}, there exists some constant $C_2$ independent of $h$ such that
\begin{equation}\label{27}
\|\Be_h\|_2\leq N_a^\frac{1}{2}\|\Be_h\|_\infty < C_2N_a^\frac{1}{2}h^p.
\end{equation}

Moreover, we introduce the notation  $\vec{\hBf}_{\mathcal{A},h}=\left[\hf_{\mathcal{A},h}(\Bx_s)\quad\hf_{\mathcal{A},h}(\Bx_{s+1})\quad\cdots\quad\hf_{\mathcal{A},h}(\Bx_{e(N)})\right]^T$ and $\vec{\Bv}:=\left[v(\Bx_s)\quad v(\Bx_{s+1})\quad\cdots\quad v(\Bx_{e(N)})\right]^T$. Then by \eqref{21}, we immediately have
\begin{equation}
J_{\ta,h}(\hf_{\mathcal{A},h})=\frac{1}{t(N)}\left\|\BA_h\vec{\hBf}_{\mathcal{A},h}-\left[\begin{array}{c}\Bc_h\\\vec{\Bq}_h\end{array}\right]\right\|_2^2=\frac{1}{t(N)}\left\|\BA_h\left(\vec{\hBf}_{\mathcal{A},h}-\vec{\Bf}\right)-\left[\begin{array}{c}\Be_h\\\vec{\Btau}_h\end{array}\right]\right\|_2^2.
\end{equation}

Since $\hf_{\mathcal{A},h}\in\mathcal{N}_{\hM}$ is a global minimizer of $J_{\ta,h}$, it satisfies $J_{\ta,h}(\hf_{\mathcal{A},h})\leq J_{\ta,h}(v)$, namely,
\begin{equation}
\frac{1}{t(N)}\left\|\BA_h\left(\vec{\hBf}_{\mathcal{A},h}-\vec{\Bf}\right)-\left[\begin{array}{c}\Be_h\\\vec{\Btau}_h\end{array}\right]\right\|_2^2\leq\frac{1}{t(N)}\left\|\BA_h\left(\vec{\Bv}-\vec{\Bf}\right)-\left[\begin{array}{c}\Be_h\\\vec{\Btau}_h\end{array}\right]\right\|_2^2,
\end{equation}
which implies
\begin{equation}
\left\|\BA_h\left(\vec{\hBf}_{\mathcal{A},h}-\vec{\Bf}\right)\right\|_2-\left\|\left[\begin{array}{c}\Be_h\\\vec{\Btau}_h\end{array}\right]\right\|_2\leq\left\|\BA_h\left(\vec{\Bv}-\vec{\Bf}\right)\right\|_2+\left\|\left[\begin{array}{c}\Be_h\\\vec{\Btau}_h\end{array}\right]\right\|_2.
\end{equation}
Therefore,
\begin{multline}\label{13}
\left\|\vec{\hBf}_{\mathcal{A},h}-\vec{\Bf}\right\|_2\leq\|\BA_h^{-1}\|_2\left\|\BA_h\left(\vec{\hBf}_{\mathcal{A},h}-\vec{\Bf}\right)\right\|_2
\leq\|\BA_h^{-1}\|_2\left(\left\|\BA_h\left(\vec{\Bv}-\vec{\Bf}\right)\right\|_2+2\left\|\left[\begin{array}{c}\Be_h\\\vec{\Btau}_h\end{array}\right]\right\|_2\right).
\end{multline}
As a consequence, by \eqref{26}, \eqref{10}, \eqref{27}, and \eqref{13}, it follows that
\begin{align*}
\left|\hf_{\mathcal{A},h}-f\right|_{2,h} & =t(N)^{-\frac{1}{2}}\left\|\vec{\hBf}_{\mathcal{A},h}-\vec{\Bf}\right\|_2\\
& \leq t(N)^{-\frac{1}{2}}\Big[\|\BA_h\|_2\|\BA_h^{-1}\|_2\|\vec{\Bv}-\vec{\Bf}\|_2+2\|\BA_h^{-1}\|_2\left(\|\Be_h\|_2^2+\|\vec{\Btau}_h\|_2^2\right)^\frac{1}{2}\Big]\\
&\leq t(N)^{-\frac{1}{2}}\Big[t(N)^\frac{1}{2}\|\BA_h\|_2\|\BA_h^{-1}\|_2\cdot e_\mathcal{A} +2\|\BA_h^{-1}\|_2\left(C_2^2N_a+C_1^2(N-M+1)\right)^\frac{1}{2}h^p\Big]\\
&\leq C\kappa_2(\BA_h)\left(h^p+e_\mathcal{A}\right).
\end{align*}
with $C$ independent of $h$ and $\mathcal{A}$, which completes the proof.
\end{proof}

The error estimate given in Theorem \ref{Thm00} is general for any types of the admissible set $\mathcal{A}$. Specifically, we propose the error estimate of the discovery using neural networks. Note that  $\mathcal{T}$ is a one-dimensional Riemannian submanifold, combining Theorem \ref{Thm00} and Proposition \ref{Pro02} directly leads to the following result.

\begin{theorem}\label{Thm01}
Under the notations and hypothesis of Theorem \ref{Thm00}, for any $J,K\in\mathbb{N^+}$ and $\delta\in(0,1)$, it satisfies:
\begin{enumerate}
\item If $f\in C(\mathcal{T}')$ and $\mathcal{N}_{\hM}$ consists of all ReLU FNNs with width $3^{d_\delta+3}\max\{d_\delta\lfloor J^{1/d_\delta}\rfloor,J+1\}$ and depth $12K+2d_\delta+14$,
\begin{equation}\label{11}
\left|\hf_{\hM,h}-f\right|_{2,h}<C\kappa_2(\BA_h)\left(h^p+e_\text{NN}(J,K)\right)
\end{equation}
with $e_\text{NN}(J,K)=\sqrt{d}\omega_f\left(4R_\mathcal{T}(1-\delta)^{-1}\sqrt{d/d_\delta}J^{-2/d_\delta}K^{-2/d_\delta}\right)$;

\item If $f\in C^r(\mathcal{T}')$ with $r\in\mathbb{N^+}$ and $\mathcal{N}_{\hM}$ consists of all ReLU FNNs with width $17r^{d_\delta+1}3^{d_\delta} d_\delta(J+2)\log_2(8J)$ and depth $18r^2(K+2)\log_2(4K)+2d_\delta$, then \eqref{11} still holds with $e_\text{NN}(J,K)=R_\mathcal{T}(r+1)^{d_\delta}8^r(1-\delta)^{-1}\|f\|_{C^r(\mathcal{T}')}J^{-2r/d_\delta}K^{-2r/d_\delta}$,
\end{enumerate}
where $d_\delta=O\left(\ln(d/\delta)/\delta^2\right)$ is an integer such that $1\leq d_\delta\leq d$; $R_\mathcal{T}$ is defined by \eqref{12}; $\omega_f(\cdot)$ is defined by \eqref{14}; $\hf_{\hM,h}\in\mathcal{N}_{\hM}$ is a global minimizer of $J_{\ta,h}$ defined by \eqref{23_2} corresponding to an LMM with order $p$; $C$ is a constant independent of $h$, $J$, $K$, $d$ and $d_\delta$. In particular, if $\kappa_2(\BA_h)$ is uniformly bounded for all $h>0$, then
\begin{equation}\label{19}
\underset{J,K\rightarrow\infty,h\rightarrow0}{lim}\left|\hf_{\hM,h}-f\right|_{2,h}=0.
\end{equation}
\end{theorem}

\begin{remark}
If $J$ and $K$ are large enough, the error bound $e_\text{NN}(J,K)$ will be overwhelmed by $h^p$. This means the LMM truncation error will dominate the network approximation error if the network size is large enough. In this situation, $\left|\hf_{\hM,h}-f\right|_{2,h}$ will decay to zero with the rate $O(h^p)$. Namely, the convergence rate has the same order as the LMM scheme.
\end{remark}

Similarly, we can develop the $\ell^2$ error estimate for the discovery on multiple trajectories \eqref{29_1}-\eqref{29_2}. It suffices to use preceding results to get an error inequality for each trajectory and take the mean square of them. Specifically, we define $|g|_{2,h,\text{multi}}=\left((t(N)N')^{-1}\sum_{n'=1}^{N'}\sum_{n=s}^{e(N)}|g(\Bx_{n,n'})|^2\right)^{1/2}$, for all $g\in C(\Omega)$,
then under the hypothesis of Theorem \ref{Thm00}, it satisfies
\begin{equation}
\left|\hf_{\hM,h}-f\right|_{2,h,\text{multi}}<C\kappa_2(\BA_h)\left(h^p+e_\mathcal{A}\right),
\end{equation}
where $\hf_{\mathcal{A},h}\in\mathcal{A}$ is a global minimizer of $J_{\ta,h}$ defined by \eqref{29_2} corresponding to an LMM with order $p$, and $e_\mathcal{A}$ is any real number such that $e_\mathcal{A}>\underset{u\in\mathcal{A}}{\inf}\underset{\Bx\in\Omega}{\sup}|u(\Bx)-f(\Bx)|$.

In particular, we can derive the error estimates for the discovery using ReLU FNNs if the governing function is either continuous or $C^r$ smooth from Proposition \ref{Pro01}. Similar arguments apply to other types of neural networks or other structures of approximations.

\subsection{Uniform Boundedness of $\kappa_2(\BA_h)$}
Next, we discuss the estimation of $\kappa_2(\BA_h)$. This is a special case, corresponding to the $|\cdot|_{2,h}$ norm, of the discussion on the stability
of LMM for dynamics discovery made in \cite{Keller2019}. Here, for completeness, we provide an alternative approach to derive a conclusion that is the same as that shown in \cite{Keller2019}. First, we introduce the following lemma (\cite{Amodio1996}),
\begin{lemma}\label{Lem01}
Given the following triangular Toeplitz band matrix
\begin{equation}\label{18}
\BT_N=\left[
\begin{array}{ccccc}
  c_0 &  &  &  &  \\
  \vdots & \ddots &  &  &  \\
  c_M & \ddots & \ddots &  &  \\
   & \ddots & \ddots & \ddots &  \\
   & & c_M & \cdots & c_0
\end{array}\right]\in\mathbb{R}^{N\times N}
\end{equation}
with $c_0\neq0$, we define the associated polynomial by $p(z)=\sum_{i=0}^{M}c_iz^{M-i}$. If all roots of $p(z)$ have modulus smaller than 1, then $\kappa_2(\BT_N)$ is uniformly bounded, i.e. $\kappa_2(\BT_N)<C$ for some $C$ independent of $N$.
\end{lemma}

Then we have the following theorem to determine the uniform boundedness of $\kappa_2(\BA_h)$,
\begin{theorem}\label{Thm03}
Let $\BA_h$ be the matrix defined by \eqref{28}, and $p_h(z)$ be the following polynomial
\begin{equation}\label{39}
p_h(z)=\sum_{i=N-e(N)}^{M-s}\beta_iz^{M-s-i}.
\end{equation}
If all roots of $p_h(z)$ have modulus smaller than 1, then $\kappa_2(\BA_h)$ is uniformly bounded with respect to $N$.
\end{theorem}
\begin{proof}
Rewrite $\BA_h$ as $2\times2$ blocks
\begin{equation}
\BA_h=\left[\begin{array}{cc}
                \BI_{N_a} & \BO\\
                \BB_{h,1} & \BB_{h,2}
              \end{array}
\right],\text{~where~}
\BB_{h,1}=\left[\begin{array}{ccc}
              \beta_{M-s} & \cdots & \beta_{N-e(N)+1} \\
               & \ddots & \vdots \\
               &  & \beta_{M-s} \\
               & \cdots &
            \end{array}
\right]\in\mathbb{R}^{(N-M+1)\times N_a},
\end{equation}
and
\begin{equation}
\BB_{h,2}=\left[
\begin{array}{ccccc}
  \beta_{N-e(N)} &  &  &  &  \\
  \vdots & \ddots &  &  &  \\
  \beta_{M-s} & \ddots & \ddots &  &  \\
   & \ddots & \ddots & \ddots &  \\
   & & \beta_{M-s} & \cdots & \beta_{N-e(N)}
\end{array}\right]\in\mathbb{R}^{(N-M+1)\times (N-M+1)}.
\end{equation}

Clearly, $\|\BB_{h,1}\|_2$ only depends on the LMM scheme and independent of $N$. By Lemma \ref{Lem01}, both $\|\BB_{h,2}\|_2$ and $\|\BB_{h,2}^{-1}\|_2$ are uniformly bounded with respect to $N$. Therefore, the proof directly follows
\begin{multline}
\|\BA_h\|_2=\underset{\|\Bx\|_2=1}{\max}\|\BA_h\Bx\|_2=\underset{\|\Bx\|_2=1}{\max}\left\|\left[\begin{array}{cc}
                \BI_{N_a} & \BO\\
                \BB_{h,1} & \BB_{h,2}
              \end{array}
\right]\left[\begin{array}{c}
                \Bx_1 \\
                \Bx_2
              \end{array}
\right]\right\|_2\\
=\underset{\|\Bx\|_2=1}{\max}\left(\|\Bx_1\|_2^2+\|\BB_{h,1}\Bx_1+\BB_{h,2}\Bx_2\|_2^2\right)^\frac{1}{2}\leq \underset{\|\Bx\|_2=1}{\max}\left(\|\Bx_1\|_2^2+\left(\|\BB_{h,1}\|_2\|\Bx_1\|_2+\|\BB_{h,2}\|_2\|\Bx_2\|_2\right)^2\right)^\frac{1}{2}\\
\leq\left(1+\left(\|\BB_{h,1}\|_2+\|\BB_{h,2}\|_2\right)^2\right)^\frac{1}{2},
\end{multline}
and
\begin{multline}
\|\BA_h^{-1}\|_2=\underset{\|\Bx\|_2=1}{\max}\|\BA_h^{-1}\Bx\|_2=\underset{\|\Bx\|_2=1}{\max}\left\|\left[\begin{array}{cc}
                \BI_{N_a} & \BO\\
                -\BB_{h,2}^{-1}\BB_{h,1} & \BB_{h,2}^{-1}
              \end{array}
\right]\left[\begin{array}{c}
                \Bx_1 \\
                \Bx_2
              \end{array}
\right]\right\|_2\\
=\underset{\|\Bx\|_2=1}{\max}\left(\|\Bx_1\|_2^2+\|-\BB_{h,2}^{-1}\BB_{h,1}\Bx_1+\BB_{h,2}^{-1}\Bx_2\|_2^2\right)^\frac{1}{2}\\
\leq\underset{\|\Bx\|_2=1}{\max}\left(\|\Bx_1\|_2^2+\|\BB_{h,2}^{-1}\|_2^2\left(\|\Bx_2\|_2+\|\BB_{h,1}\|_2\|\Bx_1\|_2\right)^2\right)^\frac{1}{2}\leq\left(1+\|\BB_{h,2}^{-1}\|_2^2\left(1+\|\BB_{h,1}\|_2\right)^2\right)^\frac{1}{2}.
\end{multline}
\end{proof}

\begin{remark}
For BDF schemes, $\beta_1=\cdots=\beta_M=0$, and the corresponding $\BB_{h,2}$ is a diagonal matrix with diagonals $\beta_0$. So $\BA_h$ is always uniformly bounded for each $M\in\mathbb{N}$. This means the network-based dynamics discovery with BDF schemes for all $M\in\mathbb{N}$ is convergent in the sense of \eqref{19}.
\end{remark}

\begin{remark}
For A-B schemes, $\BB_{h,2}$ is diagonal if $M=1$. Also, it is verified for $2\leq M\leq6$, all the roots of the associated polynomial $p_h(z)$ have
modulus smaller than 1 (\cite{Keller2019}). Hence, by Theorem \ref{Thm03}, $\BA_h$ is uniformly bounded for $1\leq m\leq6$. This means the network-based dynamics discovery with A-B schemes for $1\leq M\leq6$ is convergent in the sense of \eqref{19}.
\end{remark}

\begin{remark}\label{rmk01}
For A-M schemes with $M\geq2$, it was proven in \cite{Keller2019} that all the roots of the associated polynomial $p_h(z)$ have a modulus greater than 1. In these cases, $\kappa_2(\BA_h)$ increases exponentially with respect to $N$, and hence the error bounds in Theorem \ref{Thm01} also increases exponentially. This means we have no guarantee of their convergence in theory. In spite of this, it is still possible to obtain convergent solutions as $h\rightarrow0$ in practice (see Section \ref{Sec_case1_AM} and Appendix \ref{Sec_appendix})
\end{remark}

\begin{remark}
Note that \cite{Keller2019} considered stability under norms other than $|\cdot|_{2,h}$  as well, which also allowed the discussion of convergence for A-B family for which there are roots on the unit disc. In particular, it was shown that A-M scheme is marginally stable for $M=1$,
(see the definition in \cite{Keller2019}) but remains convergent. Actually, in this case, $\kappa_2(\BA_h)$ increases linearly with respect to $N=T/h$. If the network size is large enough such that the network approximation error is dominated by $O(h^p)$, the error bounds in Theorem \ref{Thm01} will be $C\cdot \frac{T}{h}\cdot h^p=O(h)$ since $p=2$. This means A-M scheme with $M=1$ is convergent with order 1. Moreover, Theorem \ref{Thm01} can be modified for norms other than $|\cdot|_{2,h}$ and condition number other than $\kappa_2$, resulting in various error bounds with special orders.
\end{remark}

\section{Numerical Experiments}\label{Sec_numerical}
In this section, several examples are provided to show the performance of dynamics discovery via deep learning in practical computation. We aim to compute the errors of various LMMs, estimate the orders of accuracy and compare them with the theoretical ones.

In the first, second and third examples, we conduct the discovery on a single trajectory $\mathcal{T}$ described in Section \ref{Sec_one_trajectory}, in which we define the following relative $\ell^2$ error
\begin{equation}\label{31}
e_{\hf}=\left(d^{-1}\sum_{j=1}^d\left(\sum_{n=s}^{e(N)}\left|\hf_j(\Bx_n)-f_j(\Bx_n)\right|^2\right)/\left(\sum_{n=s}^{e(N)}|f_j(\Bx_n)|^2\right)\right)^{1/2},
\end{equation}
where $f_n$ for $n=1,\cdots,d$ are components of the original governing function $\Bf$, and $\hf_n$ is the network approximating $f_n$. Note that $\{\Bx_n\}_{n=s}^{e(N)}$ are exactly the grid points involved in the loss function, the error defined by \eqref{31} is actually an empirical error. For the deep learning, we name \eqref{31} as the training error or grid error. On the other hand, we are also interested in the generalization performance of the network approximation. So we also define the relative $\ell^2$ error at testing points as

\begin{equation}\label{32}
\tilde{e}_{\hf}=\left(d^{-1}\sum_{j=1}^d\frac{\int_{\mathcal{T}}|\hf_j-f_j|^2\text{d}s}{\int_{\mathcal{T}}|f_j|^2\text{d}s}\right)^{1/2}=\left(d^{-1}\sum_{j=1}^d\frac{\int_0^T|\hf_j(\Bx(t))-f_j(\Bx(t))|^2\cdot\|\Bf(\Bx(t))\|_2\text{d}t}{\int_0^T|f_j(\Bx(t))|^2\cdot\|\Bf(\Bx(t))\|_2\text{d}t}\right)^{1/2},
\end{equation}
where the integral over $\mathcal{T}$ is evaluated by Gauss quadrature. For the deep learning, we name \eqref{32} as the testing error. Both \eqref{31} and \eqref{32} are taken as metrics for evaluation.

In the fourth example, we conduct the discovery on a compact region $\Omega$ described in Section \ref{Sec_subsets}. Similarly, we define the following training error
\begin{equation}
e_{\hf}=\left(d^{-1}\sum_{j=1}^d\frac{\sum_{n'=1}^{N'}\sum_{n=s}^{e(N)}\left|\hf_j(\Bx_{n,n'})-f_j(\Bx_{n,n'})\right|^2}{\sum_{n'=1}^{N'}\sum_{n=s}^{e(N)}|f_j(\Bx_{n,n'})|^2}\right)^{1/2},
\end{equation}
and testing error $\tilde{e}_{\hf}=\left(d^{-1}\sum_{j=1}^d\left(\int_\Omega|\hf_j-f_j|^2\text{d}\Bx\right)/\left(\int_\Omega|f_j|^2\text{d}\Bx\right)\right)^{1/2}$, where the integral over $\Omega$ is evaluated by Monte Carlo method.

The overall setting in all experiments is summarized as follows.
\begin{itemize}
  \item \textbf{Environment}
  The experiments are performed in Python 3.8 environment. We utilize PyTorch library for neural network implementation and CUDA 11.0 toolkit for GPU-based parallel computing. All examples are implemented on a desktop.
  \item \textbf{Optimizer and hyper-parameters}
  The network-based optimization is solved by {\em adam} subroutine from PyTorch library. This subroutine implements the Adam algorithm in \cite{Kingma2014}. For all examples, the number of epochs $N_\text{I}$ is set as $3\times10^4$, and use batch gradient descent. The learning rate in the $n$-th epoch, denoted as $\delta_n$, is set to decay exponentially with linearly decreasing powers from $10^{-2}$ to $10^{-4}$, namely, $\delta_n=10^{-2-2n/N_\text{I}}$.
  \item \textbf{Network setting}
   The FNN with ReLU activation is taken for approximation, whose weights and biases are initialized via uniform distributions ${\bm W}_l, {\bm b}_l\sim U(-\sqrt{W_{l-1}},\sqrt{W_{l-1}})$.
  \item \textbf{Generation of data}
   In the first example, the state data are generated directly by the explicit expression. In the second and third examples, no expression for the state is available. Hence we generate the state data by solving the dynamical system via the solver {\tt ode45} in Matlab with tiny tolerances ({\tt RelTol}$=10^{-13}$, {\tt AbsTol}$=10^{-13}$).
\end{itemize}

In the numerical implementation, the overall error is not only affected by the LMM discretization error and the network approximation error, but also by the optimization performance. In neural network optimization, it is usually difficult to find global minimizers numerically due to non-convexity. There is no existing optimizer that can guarantee to identify a global minimizer to the best of our knowledge. The optimization error is the difference between the actually identified neural network and the neural network associated with an arbitrary global minimizer. Consequently, for LMMs with uniformly bounded $\kappa_2(\BA)$, the overall error between the numerical solution and the target governing function consists of the LMM discretization error $O(h^p)$, the network approximation error determined by the network size, and the optimization error. We will validate and quantify the optimization error in our tests later.

\subsection{Problem with Accurate Data}\label{Sec_case1}
Let us consider the following model problem
\begin{equation}\label{case1}
\begin{cases}\dot{x_1}=x_2,~\dot{x_2}=-x_1,~\dot{x_3}=1/x_2^2,\quad t\in[0,1]\\\left[x_1,x_2,x_3\right]_{t=0}=[0,1,0],\end{cases},
\end{equation}
whose state can be explicitly given by $x_1=\sin(t)$, $x_2=\cos(t)$, $x_3=\tan(t)$. Thanks to the explicit expressions, we can directly take the accurate time-series $\{x_1(t_n),x_2(t_n),x_3(t_n)\}_{n=1}^N$ for the test, and no error is brought to the data. Under this setting, the error on numerical solutions are only caused by the method. In this experiment, we focus on the deep learning discovery with auxiliary initial conditions \eqref{23_1}-\eqref{23_2}.

\subsubsection{Network Size Test}
Note that Theorem \ref{Thm01} implies $e_{\hf}\sim O(h^p)$ as $h\rightarrow0$, as long as the network is sufficiently deep and wide. However, in practice, the desired depth and width are usually unknown. So we first perform the discovery with networks of various sizes to find a decent network that is both effective in approximation and cheap in computation. Specifically, we use depth $L=2$, $3$, $\cdots$, $6$, width $W=10$, $20$, $\cdots$, $2560$, and $h=10^{-3}$. The BDF-6 scheme is employed in this test. Therefore, the local truncation error is up to $O(h^6) = O(10^{-18})$, which is smaller than machine precision. Consequently, numerical errors in this case are mainly caused by network approximation (i.e., the difference of the network associated with a global minimizer of \eqref{08_1} and \eqref{23_1} and the target function) and network optimization (i.e., the difference of the networks associated with a local minimizer and a global minimizer of \eqref{08_1} and \eqref{23_1}). In Figure \ref{Fig_case1_errors_L_m}, $e_{\hf}$ and $\tilde{e}_{\hf}$ versus $W$ for various $L$ are presented. It is observed that both errors decrease quickly as $W$ increases. On the other hand, the network with $L=5$ and $W=2560$ obtains the minimal error. We can also observe that for $L=5$, the error decay becomes very slow after $W=640$. Consequently, we choose the network with $L=5$ and $W=640$ for all tests afterward, since the computation when $W=640$ is not expensive and the overall error cannot be improved significantly furthermore.

\begin{figure}
\begin{minipage}[t]{0.47\linewidth}
\centering
\subfloat[$e_{\hf}$ v.s. $W$]{
\includegraphics[scale=0.28]{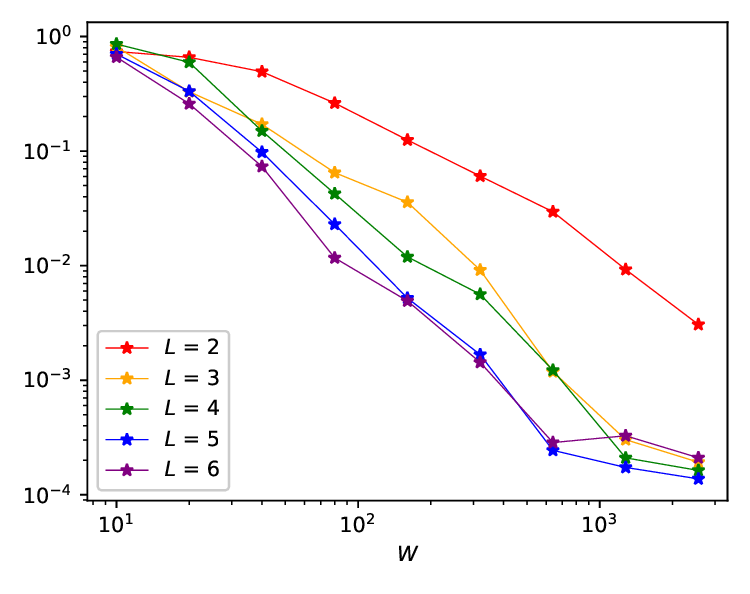}}
\subfloat[$\tilde{e}_{\hf}$ v.s. $W$]{
\includegraphics[scale=0.28]{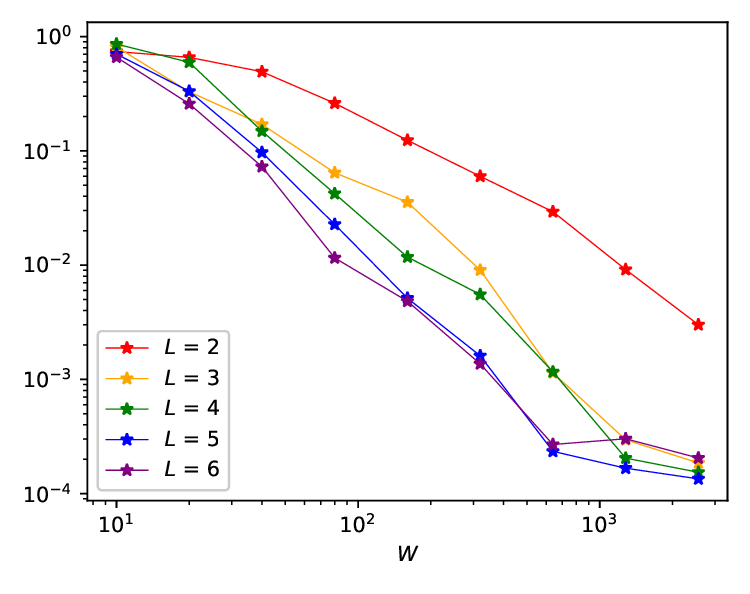}}
\caption{\em Training error $e_{\hf}$ and testing error $\tilde{e}_{\hf}$ versus $W$ of the model problem \eqref{case1}.}
\label{Fig_case1_errors_L_m}
\end{minipage}
\hfill
\begin{minipage}[t]{0.47\linewidth}
\centering
\subfloat[$e_{\hf}$ v.s. $h$]{
\includegraphics[scale=0.28]{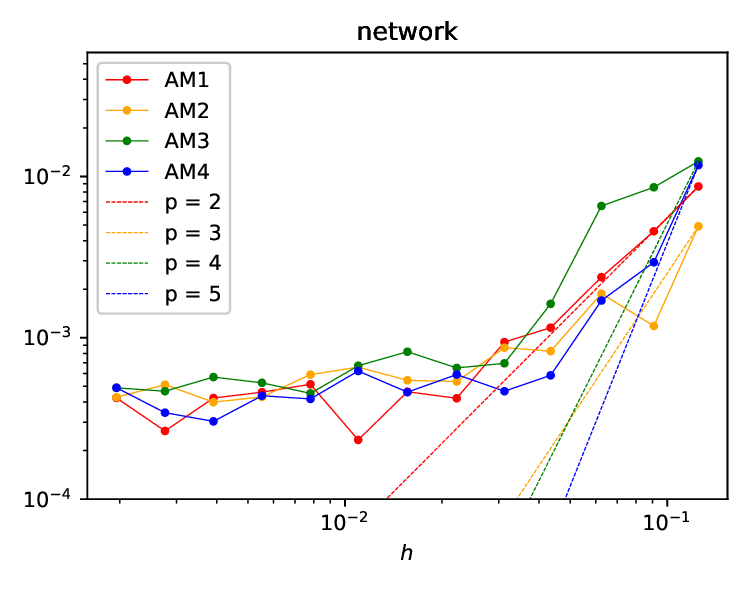}}
\subfloat[$\tilde{e}_{\hf}$ v.s. $h$]{
\includegraphics[scale=0.28]{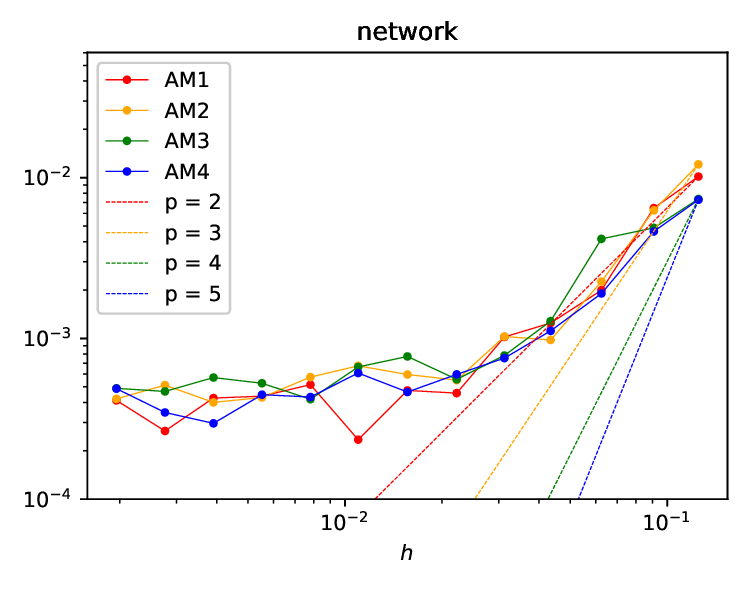}}
\caption{\em Training error $e_{\hf}$ and testing error $\tilde{e}_{\hf}$ versus $h$ via network-based A-M schemes of the model problem \eqref{case1}.}
\label{Fig_case1_errors_h_method_AM}
\end{minipage}
\end{figure}

\subsubsection{Quantification of Optimization Errors}
A special test is conducted to estimate the optimization errors. First, we set up three ReLU FNNs with $L=5$ and $W=640$, denoted as $\hat{f}^*_1$, $\hat{f}^*_2$, $\hat{f}^*_3$, and use them to fit the three components of the governing functions in \eqref{case1}, respectively. We use a standard least-square regression in this fitting. Next, we consider the dynamical system with the governing function being these FNNs, namely,
\begin{equation}\label{case1_GN}
\begin{cases}\dot{x_1}=\hat{f}^*_1,~\dot{x_2}=\hat{f}^*_2,~\dot{x_3}=\hat{f}^*_3,\quad t\in[0,1]\\\left[x_1,x_2,x_3\right]_{t=0}=[0,1,0],\end{cases}.
\end{equation}
We still use ReLU FNNs with $L=5$ and $W=640$ to do discovery on \eqref{case1_GN}. Under this setting, the approximate networks have the same architecture as the target governing function, which implies the approximation error is automatically zero. Moreover, same as the preceding test, we take BDF-6 scheme with $h=10^{-3}$, whose LMM discretization error is negligible. Therefore the obtained error should be dominated by the optimization error.

Finally, we obtain the training error $e_{\hf}=3.451\times10^{-4}$ and testing error $\tilde{e}_{\hf}=3.443\times10^{-4}$, which reflects the optimization error caused by the current optimizer is around $O(10^{-4})$. This quantification indicates that there exists an error bottleneck around $O(10^{-4})$ preventing the overall error from being reduced below it.

\subsubsection{Convergence Rate with Respect to $h$}\label{Sec_case1_AB/BDF}
Next, we test the convergence rate of the deep learning discovery by varying $h$ and using various LMM schemes. Recall the overall error consists of the LMM discretization error, the network approximation error, and the optimization error. To conduct appropriate tests on the convergence order in $h$, the network approximation error and the optimization error should be well controlled such that the LMM discretization error is the dominant error. For this purpose, we will conduct a series of tests to empirically identify a threshold $h^*>0$ such that the LMM discretization error is dominating the overall error when $h>h^*$. When $h<h^*$, although decreasing $h$ would still reduce the overall error, it is difficult to observe the order of $O(h^p)$ since, for example, the optimization error may be dominant.

Specifically, we assign $h=2^{-3},\cdots,2^{-9}$, fix the network width $W=640$, and test A-B and BDF ($M=1,\cdots,4$) schemes, both of which are proved to have uniformly bounded matrices $\BA_h$. The log-log error decay versus $h$ for each scheme is presented in Figure \ref{Fig_case1_errors_h_method}. Recall the theoretical results in Section \ref{Sec_theory} imply that the training error of the $M$-step scheme should converge to zero with order $M$. According to Figure \ref{Fig_case1_errors_h_method}, there indeed exist some empirical threshold $h^*>0$ for each scheme. It is shown in Figure \ref{Fig_case1_errors_h_method} (a) that when $h>h^*$, deep learning-based LMMs can effectively discover the governing function on training sample points with error orders close to the theoretical ones. And it is shown in Figure \ref{Fig_case1_errors_h_method} (b) that deep learning-based LMMs also have good generalization performance similar to the training error on sample points.

We would like to double-check that $W=640$ is an appropriate size and the approximation error is small enough for the convergence rate test with respect to $h$; that is, the training errors are indeed dominated by $O(h^p)$ when $h>h^*$, in which case the log-log error curves appear as straight line segments. For different $M$ and $h^*$, we repeat the preceding test using width $W=1280$ and present the new training errors in Table \ref{Tab02}. Table \ref{Tab02} shows that using $W=1280$ can not even get smaller errors in most cases, which excludes the possibility that the network approximation error is dominant. Note that decreasing $h$ can reduce errors with an expected order as long as $h>h^*$, which excludes the possibility that the optimization error is dominant. Therefore, these numerical results show that the training errors are dominated by the LMM discretization when $h>h^*$, which is a suitable range of $h$ for a convergence test.

\subsubsection{Convergence of A-M schemes}\label{Sec_case1_AM}

Moreover, we perform a test using A-M schemes with $1\leq M\leq 4$ . Although no theoretical analysis is made on the convergence of A-M schemes with $M\geq2$ (see Remark \ref{rmk01}), it is intriguing to investigate how the A-M schemes perform in practice. First, we conduct the network-based discovery with A-M schemes under the same framework as in Section \ref{Sec_case1_AB/BDF}. The training and testing errors versus $h$ are shown in Figure \ref{Fig_case1_errors_h_method_AM}. It is observed that both errors decrease as $h$ decreases, though the errors decrease more slowly when $h$ is smaller due to the optimization errors.

This result indicates that the network-based LMM with unstable schemes can still work effectively, obtaining solutions with small errors if $h$ is small enough. However, comparative tests in Appendix \ref{Sec_appendix} show that with unstable LMM schemes, using other approximations (e.g., grid functions and polynomials) are less robust, whose results are highly sensitive to the used solvers and their settings. This comparison implies that the network approximation is advantageous over other approximations in overcoming the ill-conditioning of the unstable schemes.

{Despite obtaining errors up to $O(10^{-3})$ in this test, A-M schemes are not recommended to users in practical problems. Indeed, the observed convergence rates are clearly lower than the theoretical ones, and it shows no improvement when using larger $M$. Instead, stable schemes such as A-B or BDF are more manageable in the convergence rates and not more expensive in the computational cost.}

\begin{figure}
\centering
\subfloat[$e_{\hf}$ v.s. $h$]{
\includegraphics[scale=0.3]{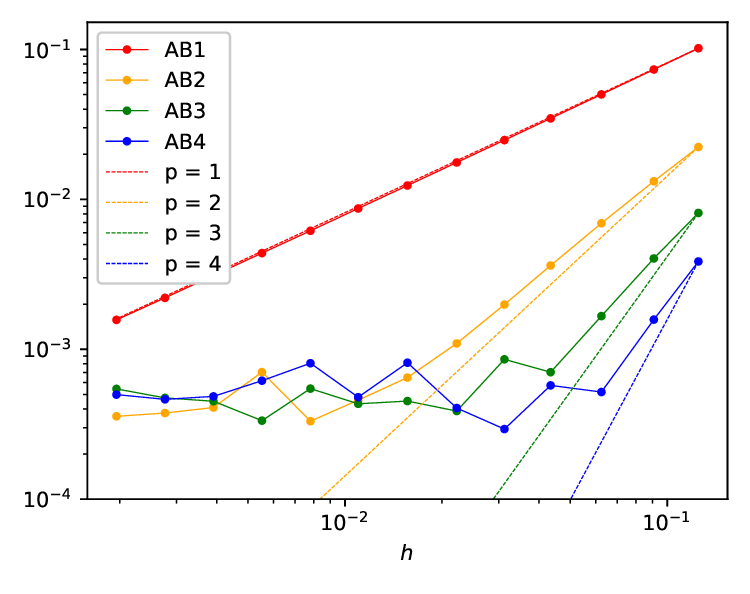}
\includegraphics[scale=0.3]{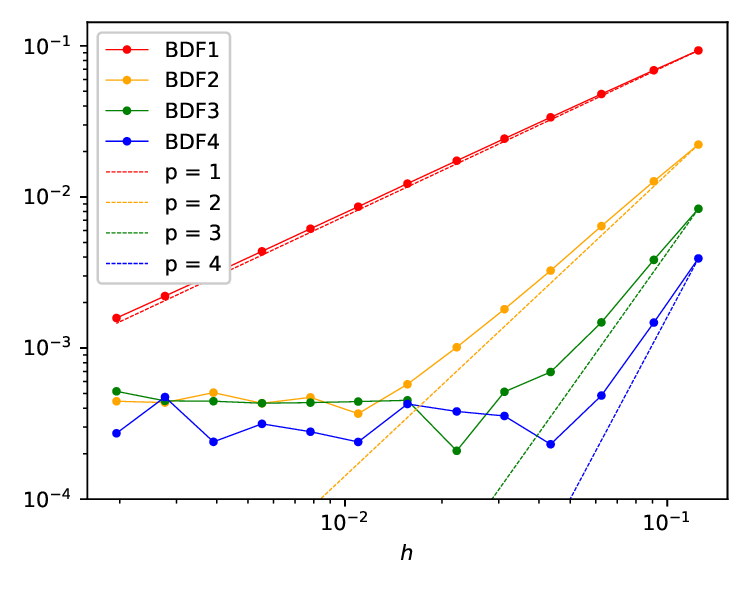}}
\subfloat[$\tilde{e}_{\hf}$ v.s. $h$]{
\includegraphics[scale=0.3]{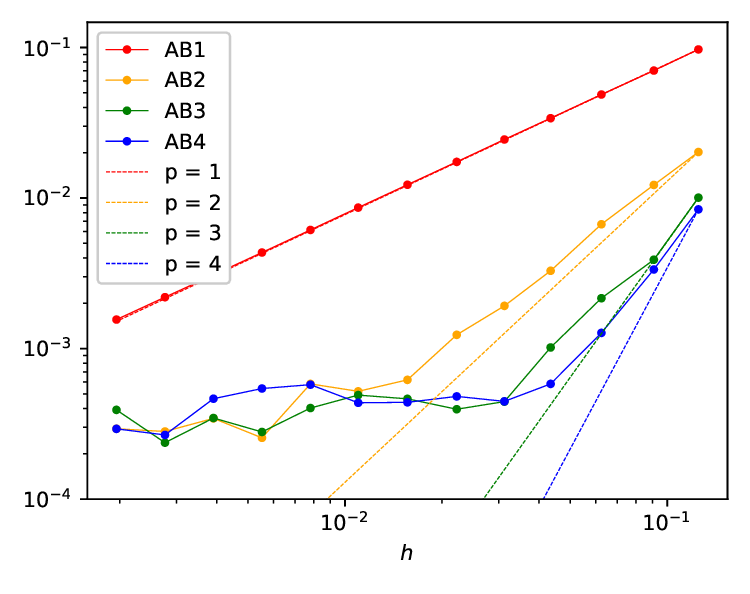}
\includegraphics[scale=0.3]{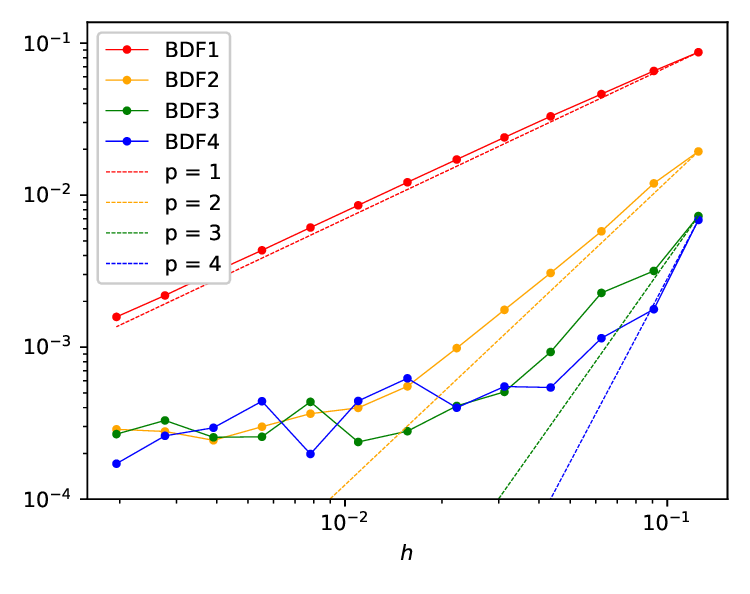}}\\
\caption{\em Training error $e_{\hf}$ and testing error $\tilde{e}_{\hf}$ versus $h$ via network-based A-B/BDF schemes of the model problem \eqref{case1}.}
\label{Fig_case1_errors_h_method}
\end{figure}

\begin{table}
\centering
\begin{tabular}{|c|c|c|c|c|c|c|c|}
  \hline
   \multicolumn{2}{|c|}{Schemes} & \multicolumn{3}{c|}{A-B} & \multicolumn{3}{c|}{BDF}  \\\hline
   $M,h^*$ &$h$ & $e_{\hf,W=640}$ & $e_{\hf,W=1280}$ & Diff & $e_{\hf,W=640}$ & $e_{\hf,W=1280}$ & Diff \\\hline
   \multirow{4}{*}{$\begin{array}{c}M=1 \\h^*=1/512 \end{array}$} &$1/8$ & 1.019e-01 & 1.019e-01 & 6.755e-10 & 9.330e-02 & 9.987e-02 & -6.576e-03 \\
   &$1/32$ & 2.485e-02 & 2.485e-02 & -8.722e-07 & 2.431e-02 & 2.471e-02 & -3.971e-04 \\
   &$1/128$ & 6.173e-03 & 6.179e-03 & -5.386e-06 & 6.139e-03 & 6.142e-03 & -3.145e-06 \\
   &$1/512$ & 1.561e-03 & 1.595e-03 & -3.420e-05 & 1.582e-03 & 1.598e-03 & -1.544e-05 \\\hline
   \multirow{4}{*}{$\begin{array}{c}M=2 \\h^*=1/64 \end{array}$} &$1/8$ & 2.234e-02 & 2.234e-02 & -1.937e-12 & 2.222e-02 & 2.222e-02 & 1.128e-10 \\
   &$1/16$ & 6.928e-03 & 6.930e-03 & -1.946e-06 & 6.405e-03 & 1.433e-02 & -7.928e-03 \\
   &$1/32$ & 1.987e-03 & 2.045e-03 & -5.865e-05 & 1.746e-03 & 1.877e-03 & -1.304e-04 \\
   &$1/64$ & 6.110e-04 & 8.278e-04 & -2.167e-04 & 5.380e-04 & 1.298e-03 & -7.602e-04 \\\hline
   \multirow{4}{*}{$\begin{array}{c}M=3 \\h^*=1/23 \end{array}$} &$1/8$ & 8.117e-03 & 8.117e-03 & -1.388e-17 & 8.354e-03 & 8.354e-03 & 1.105e-08 \\
   &$1/11$ & 4.035e-03 & 4.036e-03 & -1.016e-06 & 3.836e-03 & 3.840e-03 & -4.602e-06 \\
   &$1/16$ & 1.665e-03 & 1.744e-03 & -7.959e-05 & 1.471e-03 & 1.950e-03 & -4.786e-04 \\
   &$1/23$ & 6.735e-04 & 7.855e-04 & -1.120e-04 & 6.531e-04 & 8.099e-03 & -7.446e-03 \\\hline
   \multirow{3}{*}{$\begin{array}{c}M=4 \\h^*=1/16 \end{array}$} &$1/8$ & 3.852e-03 & 3.949e-03 & -9.610e-05 & 3.920e-03 & 3.920e-03 & 1.996e-14 \\
   &$1/11$ & 1.577e-03 & 1.577e-03 & -4.042e-09 & 1.472e-03 & 1.470e-03 & 2.552e-06 \\
   &$1/16$ & 5.179e-04 & 6.443e-04 & -1.263e-04 & 5.413e-04 & 1.375e-03 & -8.338e-04 \\\hline
\end{tabular}
\caption{\em  Training errors $e_{\hf}$ with $W=640$ and $1280$ for various $h$'s using A-B and BDF schemes of the model problem \eqref{case1}. $h^*$ denotes the step size threshold above which the error decreases as $O(h^p)$ approximately. ``Diff" denotes $e_{\hf,W=640}-e_{\hf,W=1280}$.}\label{Tab02}
\end{table}

\subsubsection{Variability test}\label{Sec_case1_variability}
{
Finally, we conduct a variability test by repeating the experiments with randomness. Note that the randomness of our algorithm only comes from the initialization of neural networks. In this test, the A-B, BDF and A-M schemes for various $M$ and $h$ are implemented repeatedly on the model problem \eqref{case1} using 10 different random seeds, and we compute the average errors and their standard deviations of these trials. Selected results for the training errors are presented in Table \ref{Tab03}. It is clear that most of the average errors dominate their standard deviations, and in some cases they have the same magnitude (e.g., $(M,h)=(4,1/32)$). Same results apply to the testing errors. Therefore the computed errors with any random seeds are kept in the same magnitude with high probability. Consequently, our algorithm is numerically stable under the random initialization, and hence all experiments and conclusions are reliable.

\begin{table}
\centering
\begin{tabular}{|c|c|c|c|c|c|c|}
  \hline
   Schemes & \multicolumn{2}{c|}{A-B} & \multicolumn{2}{c|}{BDF} & \multicolumn{2}{c|}{A-M} \\\hline
    $(M,h)$ & $e_{\hf}$ & SD & $e_{\hf}$ & SD & $e_{\hf}$ & SD \\\hline
    $(1,1/8)$ & 1.019e-01 & 2.789e-09 & 9.330e-02 & 5.626e-09 & 9.007e-03 & 3.433e-04 \\\hline
    $(1,1/16)$ & 5.032e-02 & 2.381e-04 & 4.801e-02 & 2.707e-07 & 2.258e-03 & 7.404e-05 \\\hline
    $(1,1/32)$ & 2.489e-02 & 5.792e-05 & 2.432e-02 & 6.625e-07 & 1.262e-03 & 8.316e-04 \\\hline
    $(1,1/64)$ & 1.237e-02 & 3.607e-06 & 1.224e-02 & 1.376e-05 & 3.968e-04 & 3.892e-05 \\\hline
    $(1,1/128)$ & 6.175e-03 & 3.502e-06 & 6.145e-03 & 4.452e-06 & 4.099e-04 & 9.328e-05 \\\hline
    $(4,1/8)$ & 3.852e-03 & 2.084e-09 & 3.920e-03 & 3.610e-12 & 1.059e-02 & 1.919e-03 \\\hline
    $(4,1/16)$ & 5.183e-04 & 1.146e-06 & 4.742e-04 & 1.757e-05 & 3.492e-03 & 1.472e-03 \\\hline
    $(4,1/32)$ & 1.923e-03 & 1.691e-03 & 6.707e-04 & 5.927e-04 & 9.407e-04 & 5.669e-04 \\\hline
    $(4,1/64)$ & 9.154e-04 & 5.227e-04 & 3.283e-04 & 5.869e-05 & 6.474e-04 & 4.152e-04 \\\hline
    $(4,1/128)$ & 4.678e-04 & 9.506e-05 & 2.736e-04 & 4.330e-05 & 3.945e-04 & 1.279e-04 \\\hline
\end{tabular}
\caption{\em  Average training errors $e_{\hf}$ and standard deviations (SDs) of 10 trials with different random seeds. (Used network size: $L=5$, $W=640$)}\label{Tab03}
\end{table}
}

\subsection{Lorenz System}\label{Sec_case2}
Let us consider the 3-D Lorenz system which characterizes the chaotic dynamics for certain initial conditions and has a number of important applications including weather forecasting. The system is formulated as
\begin{equation}\label{case2}
\dot{x_1}=10(x_2-x_1),~\dot{x_2}=x_1(28-x_3)-x_2,~\dot{x_3}=x_1x_2-8x_3/3,\quad t\in[0,T],
\end{equation}

\subsubsection{Convergence Rate with Respect to $h$}
We continue testing the convergence rate with respect to $h$ of the dynamics discovery via deep learning. As in the previous convergence test, the test is only valid when $h$ is larger than a threshold $h^*$ when the LMM discretization error is dominating the overall error. For simplicity, we only empirically choose $W=640$ since this width is large enough for the previous test. Specifically, we consider the long time behavior of the system \eqref{case2} by setting $T=25$ and taking initial values $\left[x_1,x_2,x_3\right]_{t=0}=[-8,7,27]$. We assign $h=0.04$, $0.02$, $\cdots$, $0.0025$ and take A-B ($M=1,\cdots,4$) and BDF ($M=1,\cdots,4$) schemes. First, we conduct the optimization with initial conditions \eqref{23_1}-\eqref{23_2}. The error decay versus $h$ is demonstrated in Figure \ref{Fig_case2_errors_h_method}. The dynamics of the true governing function and the approximate neural network obtained by A-B ($M=4$, $h=0.0025$) are also presented in Figure \ref{Fig_case2_solution}, from which we observe that deep learning can identify the chaotic dynamics on training samples effectively.

As discussed in Section \ref{Sec_implicit_regularization}, it is conjectured that auxiliary conditions may not be necessary to guarantee a correct solution because the implicit regularization has a bias towards the smoothest solution. To validate this fact, we conduct a comparative test by solving the optimization \eqref{08_1}-\eqref{08_2} with or without auxiliary conditions (ACs) in the loss function of the problem in \eqref{case2}. We take the same parameters as in the preceding test and visualize the error decay versus $h$ for A-B schemes in Figure \ref{Fig_case2_errors_h_method}. We visualize the training error and loss versus training iterations in Figure \ref{Fig_case2_learning_information}. {It is clear that when $h=0.02$, the error of the A-B ($M=2$) scheme without auxiliary conditions is larger than the one with initial auxiliary conditions. The difference is also significant for A-B scheme ($M=4$) with $h\leq0.01$. The comparison shows that the approach with auxiliary conditions is more accurate, although both approaches work effectively overall.} Due to the non-uniqueness of networks approximately minimizing the loss function, networks with and without the auxiliary conditions can both reduce the loss functions well as shown by Figure \ref{Fig_case2_learning_information} (c). However, reducing the loss function well does not imply the corresponding network converges to the right target function. When $h$ is large, though the implicit regularization of deep learning can provide a smooth solution without auxiliary conditions, this solution may not be our target function and, hence, the error $e_{\hf}$ on the training grid points and the error $\tilde{e}_{\hf}$ on random grid points would be large as shown in Figure \ref{Fig_case2_learning_information} (a) and (b) (left). When $h$ is small, a larger number of training samples makes the loss function better restrict its local minimizers closer to the desired solution and, hence, both $e_{\hf}$ and $\tilde{e}_{\hf}$ becomes reasonable. The auxiliary conditions can
better eliminate spurious local minimizers of the loss function and, hence, both $e_{\hf}$ and $\tilde{e}_{\hf}$ are reasonably small no matter $h$ is large or small as shown in Figure \ref{Fig_case2_learning_information} (a) and (b) (right).

\begin{figure}
\centering
\subfloat[$e_{\hf}$ v.s. $h$]{
\includegraphics[scale=0.3]{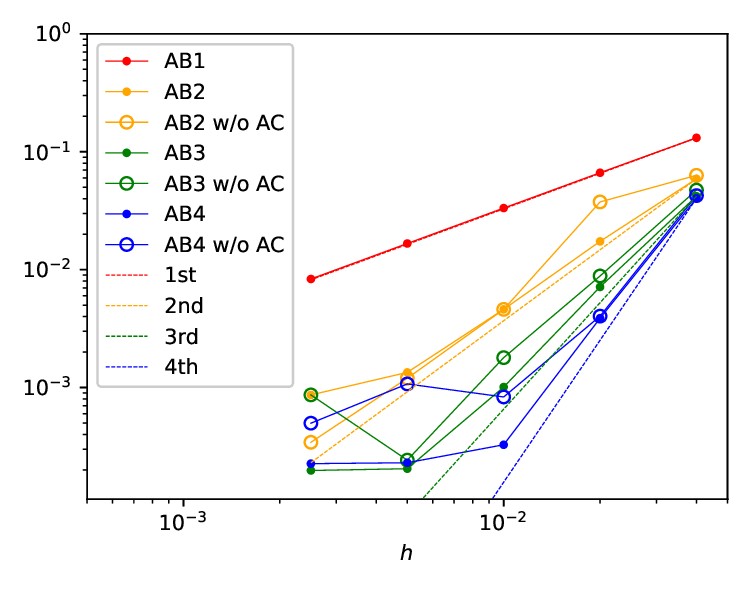}
\includegraphics[scale=0.3]{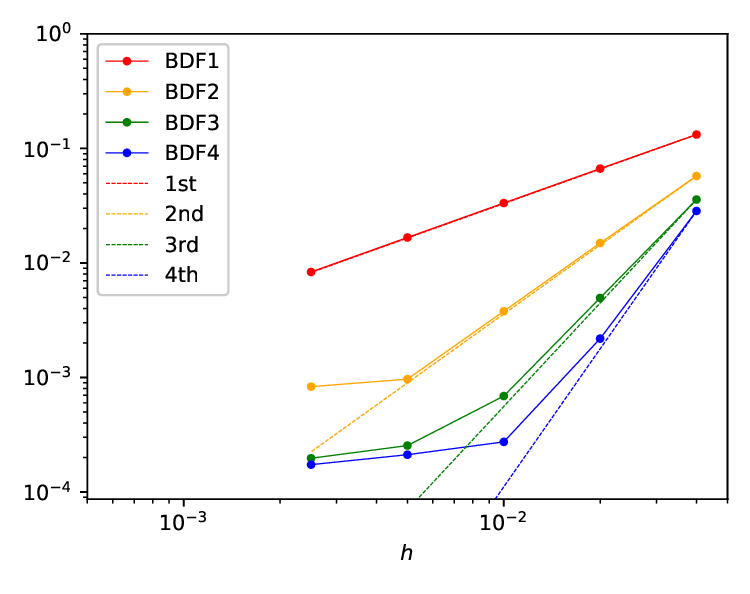}}
\subfloat[$\tilde{e}_{\hf}$ v.s. $h$]{
\includegraphics[scale=0.3]{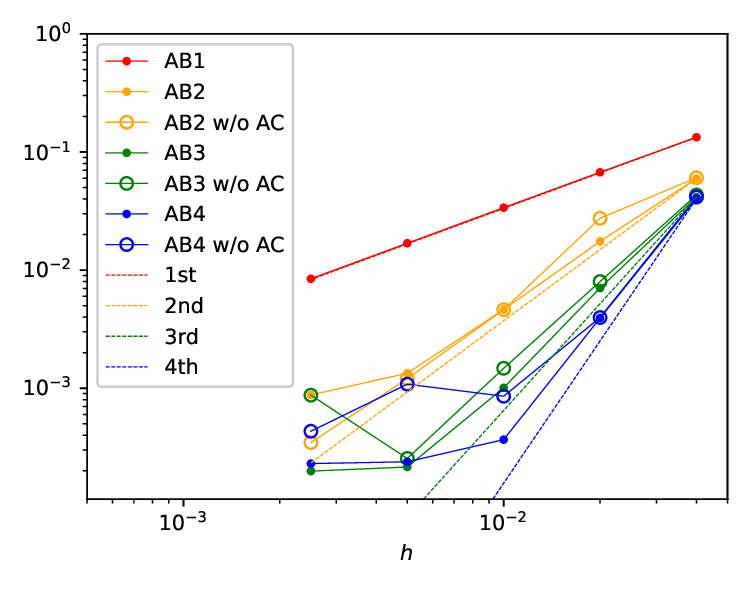}
\includegraphics[scale=0.3]{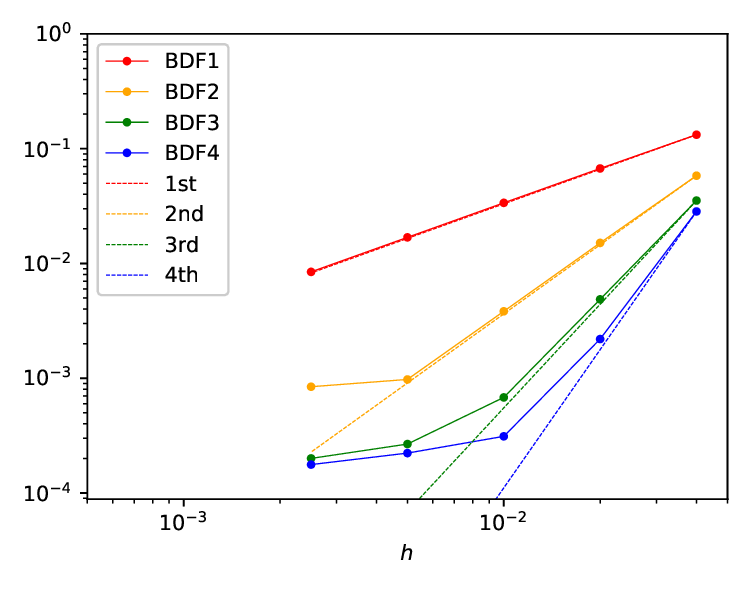}}\\
\caption{\em Training error $e_{\hf}$ and testing error $\tilde{e}_{\hf}$ versus $h$ via network-based A-B/BDF schemes with or without auxiliary conditions (ACs) of Lorenz system \eqref{case2}.}
\label{Fig_case2_errors_h_method}
\end{figure}

\begin{figure}
\begin{minipage}[t]{0.47\linewidth}
\centering
\includegraphics[scale=0.35]{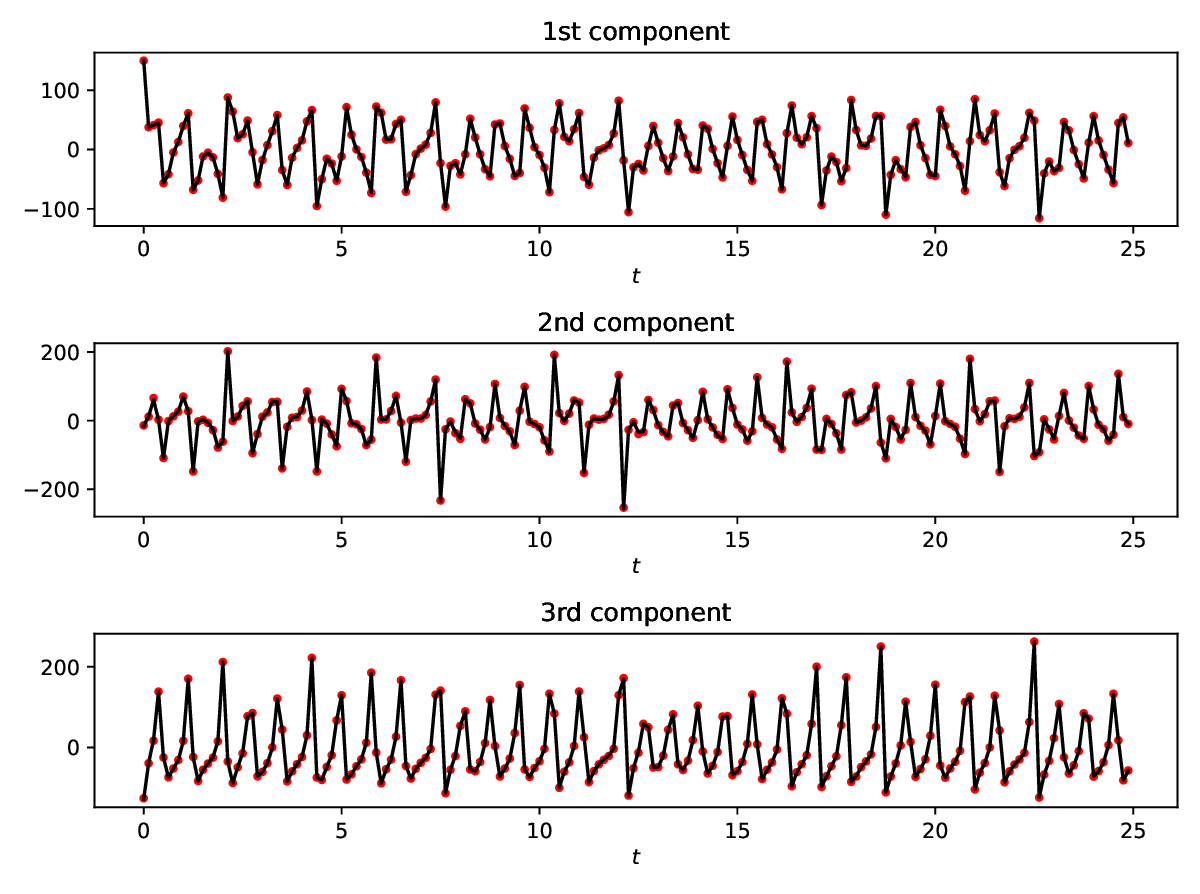}
\caption{\em The true governing function (solid black curves) and the approximate neural network (red circles) of Lorenz system \eqref{case2}.}
\label{Fig_case2_solution}
\end{minipage}
\hfill
\begin{minipage}[t]{0.47\linewidth}
\includegraphics[scale=0.4]{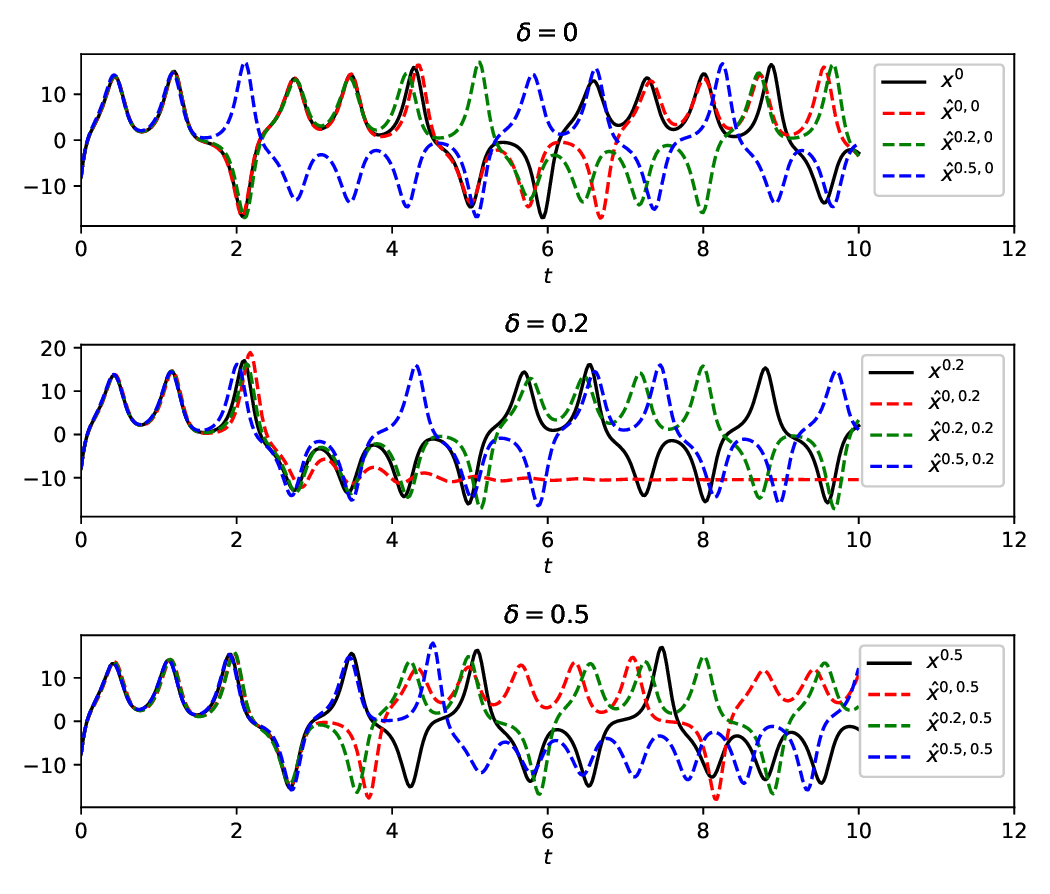}
\caption{\em The states $\Bx^\delta$ of the exact dynamics \eqref{case2} (solid black curves) and the states $\hat{\Bx}^{\varepsilon,\delta}$ of the discovered dynamics (dashed red curves for $\varepsilon=0$, dashed green curves for $\varepsilon=0.2$ and dashed blue curves for $\varepsilon=0.5$) with various initial values $[-8,7,27]+\delta$, $\delta=0$, $0.2$, $0.5$, in Lorenz system \eqref{case2}.}
\label{Fig_case2_prediction}
\end{minipage}
\end{figure}

\begin{figure}
\centering
\subfloat[$e_{\hf}$ v.s. iterations]{
\includegraphics[scale=0.36]{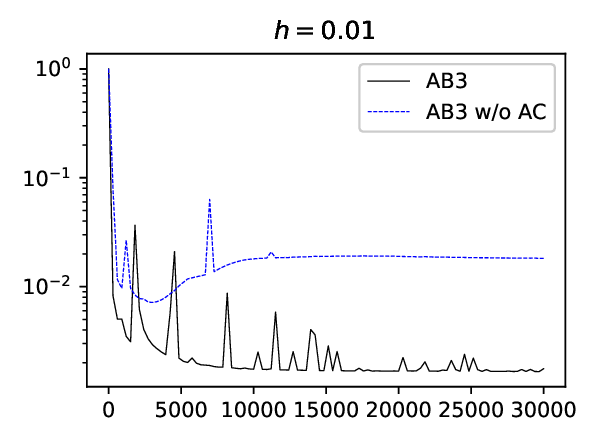}
\includegraphics[scale=0.36]{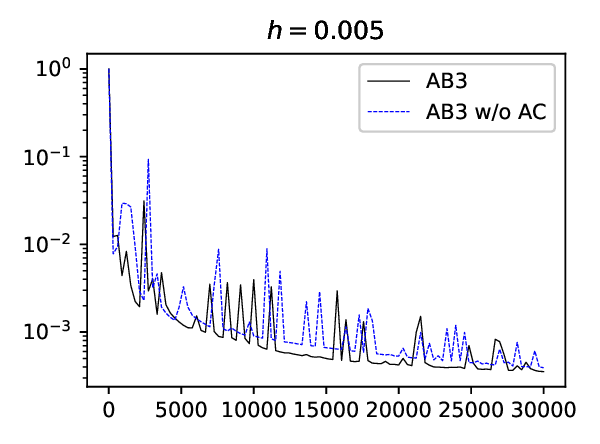}}
\subfloat[$\tilde{e}_{\hf}$ v.s. iterations]{
\includegraphics[scale=0.36]{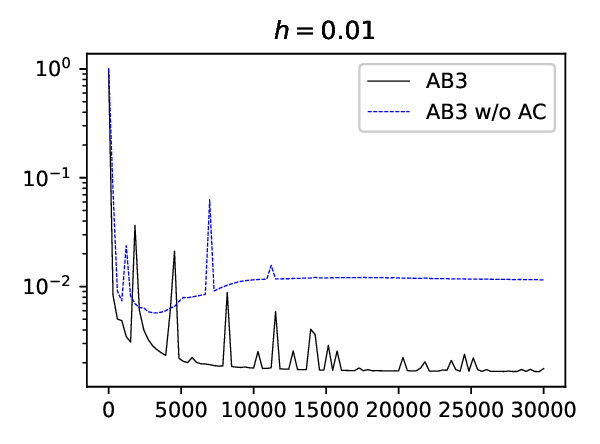}
\includegraphics[scale=0.36]{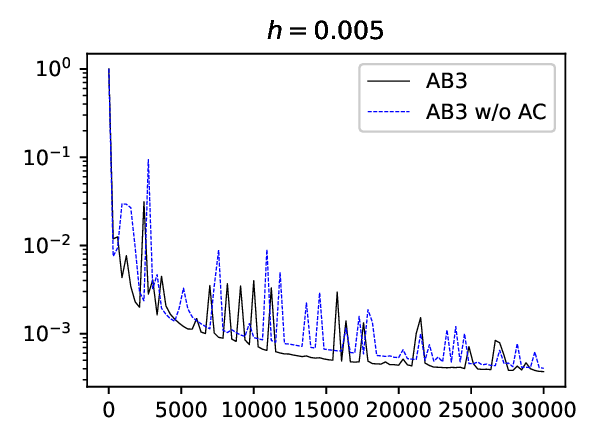}}\\
\subfloat[loss function v.s. iterations]{
\includegraphics[scale=0.36]{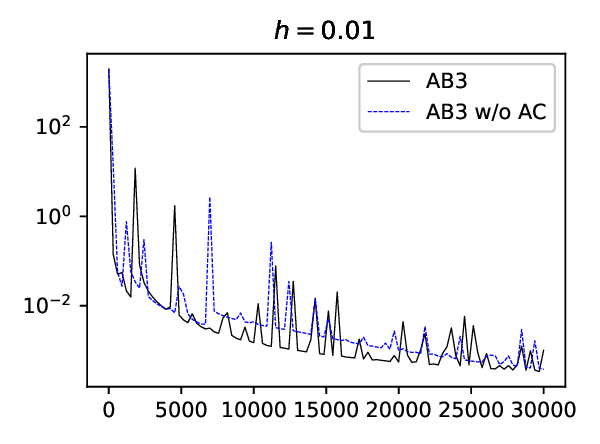}
\includegraphics[scale=0.36]{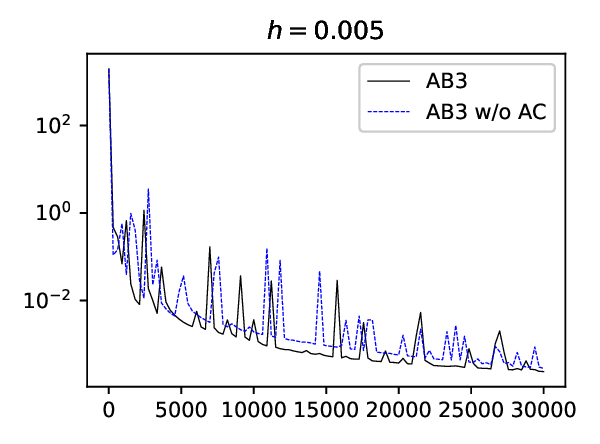}}
\caption{\em The training error $e_{\hf}$, testing error $\tilde{e}_{\hf}$ and loss function versus training iterations via network-based A-B schemes with or without auxiliary conditions (ACs) of Lorenz system \eqref{case2}.}
\label{Fig_case2_learning_information}
\end{figure}

\subsubsection{Prediction}
In real applications, we are interested in how well the discovered dynamics perform in making predictions. For this purpose, we first discover the system \eqref{case2} with initial values $[-8,7,27]+\varepsilon$ for $\varepsilon=0$, $0.2$ and $0.5$ by A-B scheme ($M=4$, $h=0.0025$), obtaining networks $\hat{\Bf}^0$, $\hat{\Bf}^{0.2}$ and $\hat{\Bf}^{0.5}$, respectively. Next, we solve the discovered system $\dot{\Bx}=\hat{\Bf}^\varepsilon$ with initial values $[-8,7,27]+\delta$ for $\delta=0$, $0.2$ and $0.5$ by the Matlab solver {\tt ode45} with tiny tolerances ({\tt RelTol}$=10^{-13}$, {\tt AbsTol}$=10^{-13}$), obtaining states $\hat{\Bx}^{\varepsilon,\delta}(t)$. Moreover, we compute the states of the exact system \eqref{case2} with initial values $[-8,7,27]+\delta$, denoting as $\Bx^\delta(t)$, for comparison. The first component of these states are presented in Figure \ref{Fig_case2_prediction}.

It can be observed that the predicted states $\hat{\Bx}^{\varepsilon,\delta}(t)$ become less accurate over time and ultimately fail to match the true states after a particular time. The inaccuracy of the long-time prediction for the Lorenz system is imputed to its positive Lyapunov exponent, which results in exponential growth of any tiny initial error over time \cite{Wang2008,Vaidyanathan2014}. Furthermore, the prediction performance also depends on the discrepancy between the initial value in prediction and the one for learning. Smaller discrepancy leads to better prediction. For example, in the case of $\delta=0$, it is shown that $\hat{\Bx}^{0,0}(t)$ moves consistently with $\Bx^0(t)$ until $t=5.1$, while $\hat{\Bx}^{0,0.2}(t)$ and $\hat{\Bx}^{0,0.5}(t)$ can only keep the consistency before $t=4.5$ and $t=1.5$, respectively. Similarly, for $\delta=0.2$ and $0.5$, the states $\hat{\Bx}^{\delta,\delta}(t)$ has a longer accurately predicted period than $\hat{\Bx}^{\varepsilon,\delta}(t)$ with $\varepsilon\neq\delta$. These numerical observations are due to the fact that only training samples of one trajectory are provided in deep learning and, hence, the recovered force term may not be accurate far away from the sampled trajectory.

\subsection{Glycolytic Oscillator}\label{Sec_case3}
We consider the model of oscillations in yeast glycolysis, which is a nonlinear biological system \cite{Daniels2015}. The model concentrates on 7 biochemical species:
\begin{equation}\label{case3}
\begin{cases}
\dot{S}_1=J_0-\frac{k_1S_1S_6}{1+(S_6/K_1)^q},\\
\dot{S}_2=2\frac{k_1S_1S_6}{1+(S_6/K_1)^q}-k_2S_2(N-S_5)-k_6S_2S_5,\\
\dot{S}_3=k_2S_2(N-S_5)-k_3S_3(A-S_6),\\
\dot{S}_4=k_3S_3(A-S_6)-k_4S_4S_5-\kappa(S_4-S_7),\\
\dot{S}_5=k_2S_2(N-S_5)-k_4S_4S_5-k_6S_2S_5,\\
\dot{S}_6=-2\frac{k_1S_1S_6}{1+(S_6/K_1)^q}+2k_3S_3(A-S_6)-k_5S_6,\\
\dot{S}_7=\psi\kappa(S_4-S_7)-kS_7,
\end{cases}\quad t\in[0,T],
\end{equation}
where the model parameters are taken from Table 1 in \cite{Daniels2015}.

\subsubsection{Convergence Rate Test with Respect to $h$}
We continue testing the convergence rate with respect to $h$ on the long time behavior of the system \eqref{case3} with $T=10$ and the initial value $[S_1,S_2,S_3,S_4,S_5,S_6,S_7]_{t=0}=\BS_0$, where $\BS_0=[1.125,0.95,0.075,0.16,0.265,0.7,0.092]$.

Similar to the preceding case, we assign $h=0.04$, $0.02$, $\cdots$, $0.04/2^6$ and conduct the optimization \eqref{23_1}-\eqref{23_2} with A-B ($M=1,\cdots,4$) and BDF ($M=1,\cdots,4$) schemes. The error decay versus $h$ is demonstrated in Figure \ref{Fig_case3_errors_h_method}. The dynamics of the true governing function and the neural network approximation obtained by A-B ($M=4$, $h=0.04/2^6$) are presented in Figure \ref{Fig_case3_solution}. It is observed that when $h$ is relatively large, the numerical convergence rates of all schemes are much lower than the theoretical ones. One explanation is that the low regularity of this system worsens the accuracy of LMMs. In Figure \ref{Fig_case3_solution}, it is clear that the governing function appears highly oscillatory with only $C^0$ regularity. Even in this challenging case, high-order LMM schemes can still recover the governing function up to $O(10^{-3})$ accuracy as $h$ decreases.

\begin{figure}
\centering
\subfloat[$e_{\hf}$ v.s. $h$]{
\includegraphics[scale=0.3]{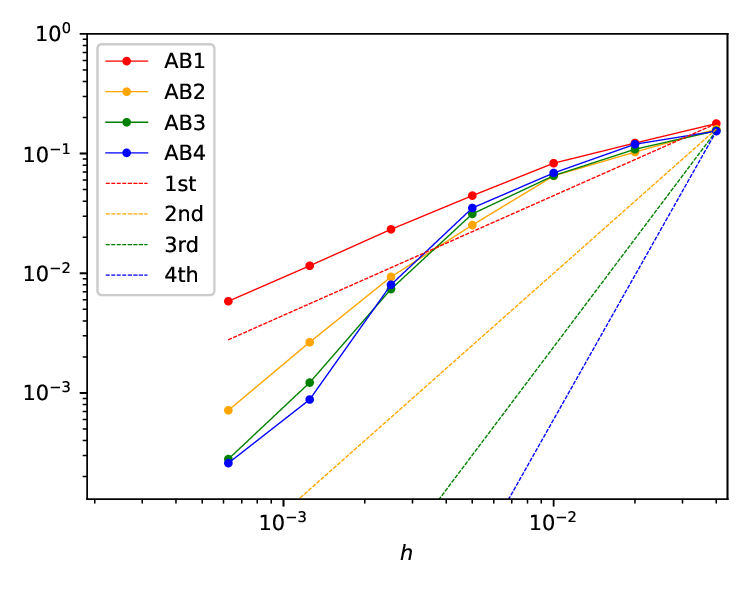}
\includegraphics[scale=0.3]{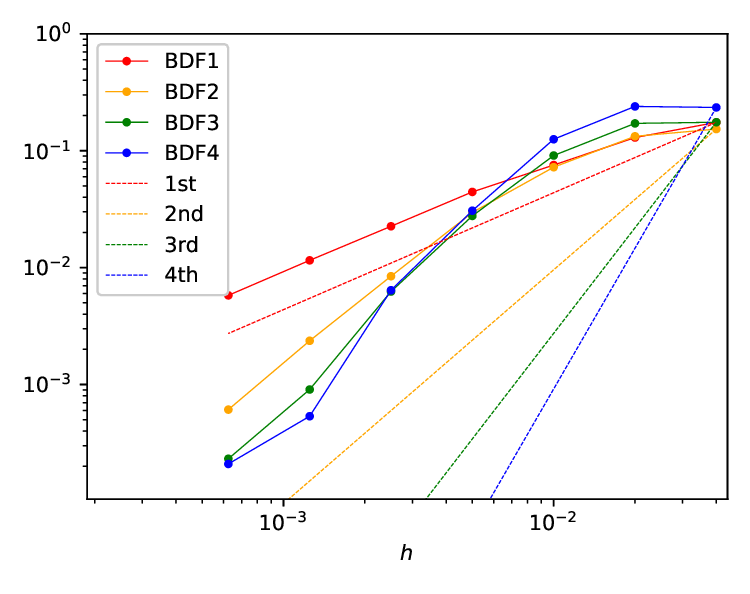}}
\subfloat[$\tilde{e}_{\hf}$ v.s. $h$]{
\includegraphics[scale=0.3]{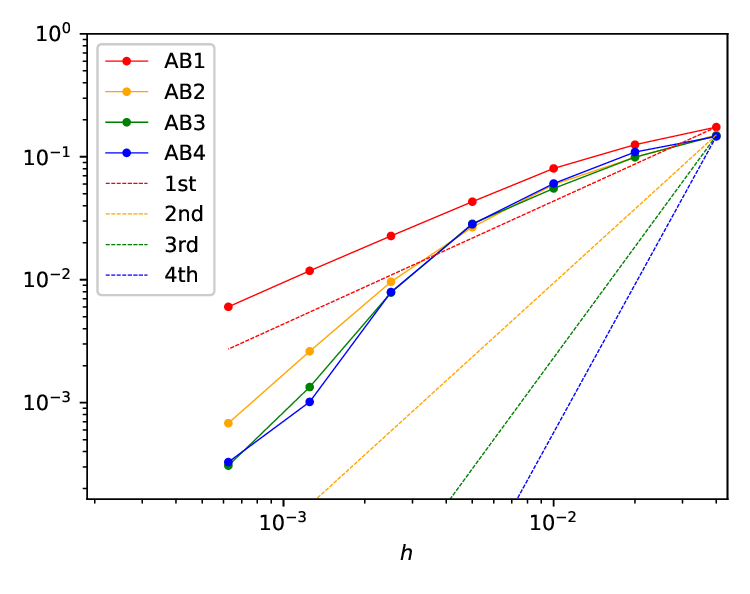}
\includegraphics[scale=0.3]{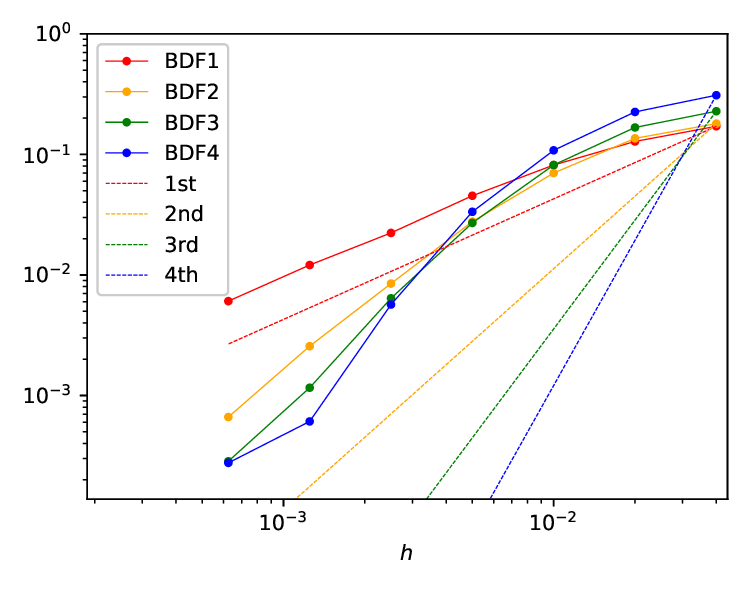}}\\
\caption{\em Training error $e_{\hf}$ and testing error $\tilde{e}_{\hf}$ versus $h$ via network-based A-B/BDF schemes of Glycolytic oscillator \eqref{case3}}
\label{Fig_case3_errors_h_method}
\end{figure}

\begin{figure}
\includegraphics[scale=0.5]{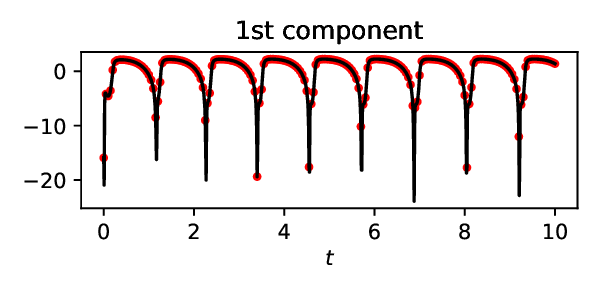}
\includegraphics[scale=0.5]{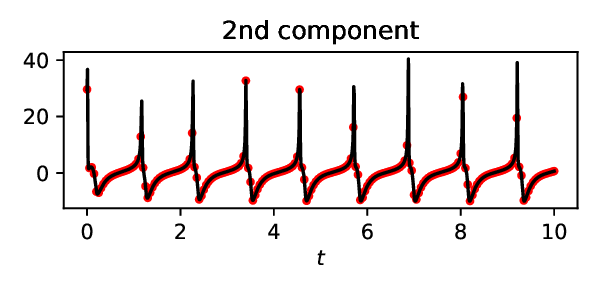}
\includegraphics[scale=0.5]{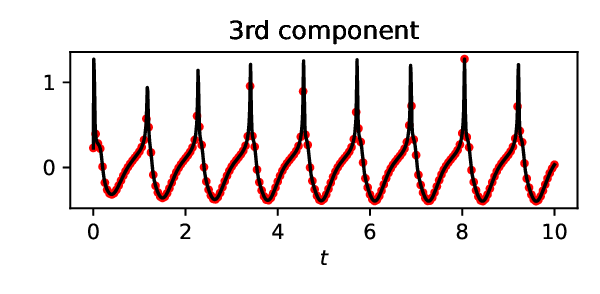}\\
\includegraphics[scale=0.5]{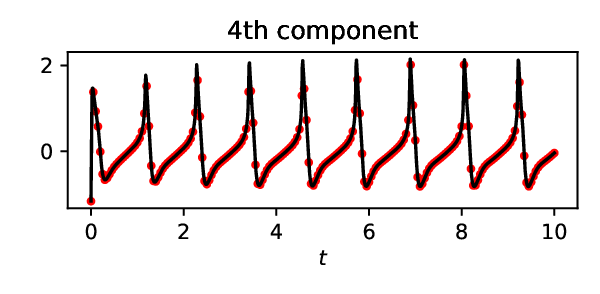}
\includegraphics[scale=0.5]{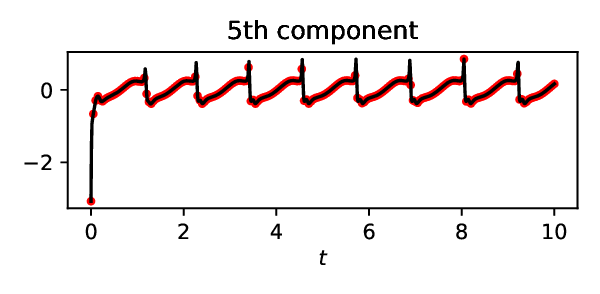}
\includegraphics[scale=0.5]{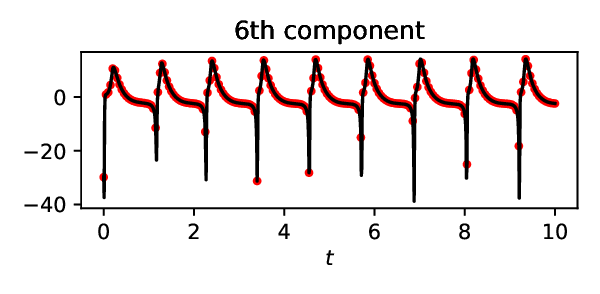}
\caption{\em The true governing function (solid black curves) and the approximate neural network (red circles) of Glycolytic oscillator \eqref{case3}.}
\label{Fig_case3_solution}
\end{figure}

\subsubsection{Prediction}
Similar to the preceding example, a prediction test is conducted for the glycolytic oscillator system. We compare the states of the exact system \eqref{case3} and the system discovered by the A-B scheme ($M=4$, $h=0.00125$) with training data generated with the initial value $\BS_0$. The states are computed with initial values $\BS_0+\delta$ for $\delta=0$, $0.05$ and $0.2$. In Figure \ref{Fig_case3_prediction}, we present the first component of states. The overall prediction performance in this example is better than that of the chaotic Lorenz system. The forecast time-series when $\delta=0$ is very accurate. The forecast time-series when $\delta=0.05$ and $0.2$ are also reasonably accurate, though the prediction error is obvious when the prediction time is large.

\begin{figure}
\begin{minipage}[t]{0.47\linewidth}
\centering
\includegraphics[scale=0.4]{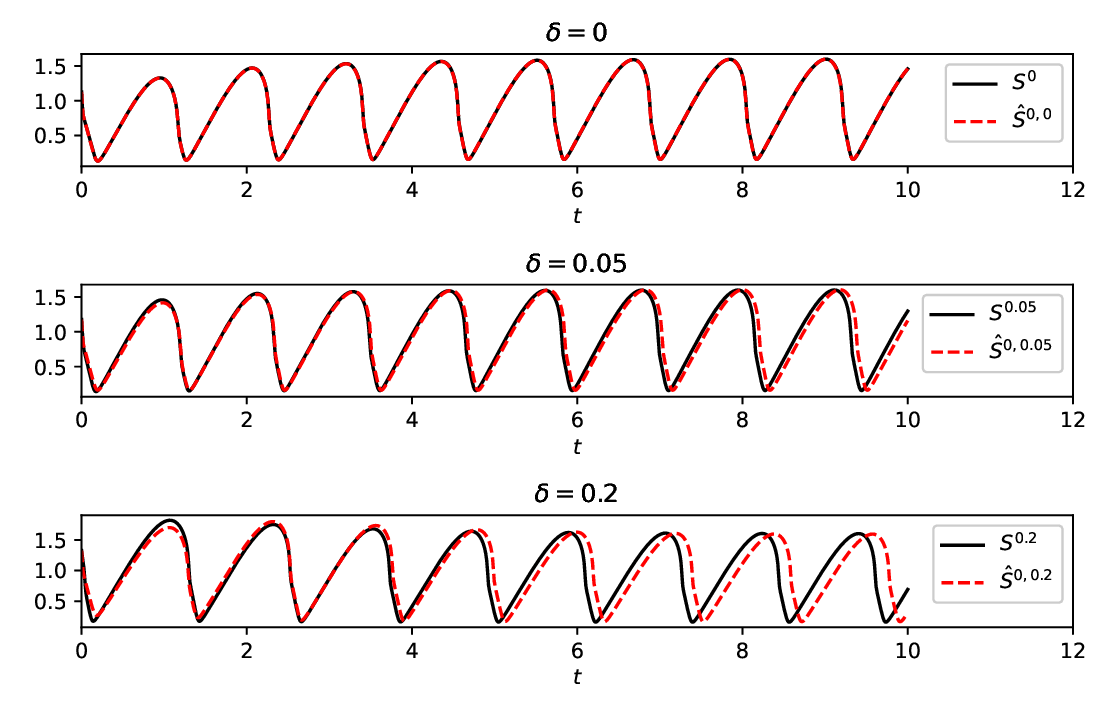}
\caption{\em The states $\BS^\delta$ of exact dynamics \eqref{case3} (solid black curves) and the states $\hat{\BS}^{0,\delta}$ of discovered dynamics (dashed red curves) for initial values $\BS_0+\delta$ with $\delta=0$, $0.05$, $0.2$, in Glycolytic oscillator \eqref{case3}.}
\label{Fig_case3_prediction}
\end{minipage}
\hfill
\begin{minipage}[t]{0.47\linewidth}
\centering
\includegraphics[scale=0.6]{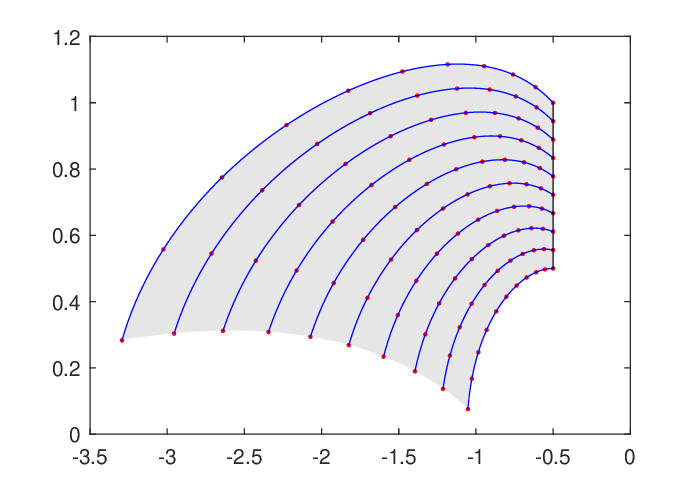}
\caption{\em State points (red points) collected from a sequence of trajectories (blue curves), which are computed from the model system \eqref{case4} with initial points equidistantly distributed on the line segment $\Gamma$ (black) in the model system \eqref{case4}}
\label{Fig_case4_domain}
\end{minipage}
\end{figure}

\subsection{Discovery on a Compact Region}
In this example, we consider the following model system
\begin{equation}\label{case4}
\begin{cases}\dot{x_1}=2x_1x_2,~\dot{x_2}=x_1+x_2,\quad t\in[0,1]\\\left[x_1,x_2\right]_{t=0}=\tilde{\Bx},\end{cases}.
\end{equation}
The initial value point $\tilde{\Bx}$ is chosen from the line segment $\Gamma=\{(-0.5,x_2):0.5\leq x_2\leq1\}$. All the trajectories starting from $\Gamma$ within $t\in[0,1]$ will form a compact region in $\mathbb{R}^2$, denoted as $\Omega$. Note that $\Omega$ is enclosed with $\Gamma$, $\{(x_1(1;\tilde{\Bx}),x_2(1;\tilde{\Bx})):\tilde{\Bx}\in\Gamma)\}$ and two outside trajectories. We collect the data of discrete states in $\Omega$. Specifically, we choose $N'$ points $\tilde{\Bx}_1,\cdots,\tilde{\Bx}_{N'}$ by equidistantly partitioning $\Gamma$ as the initial values. Next, we compute the trajectories $\Bx(t;\tilde{\Bx}_{n'})$ for $n'=1,\cdots,N'$ and take $\{\Bx(t_n;\tilde{\Bx}_{n'})\}_{n=0,\cdots,N;n'=1\cdots,N'}$ as the dataset. To display the data sampling clearly, we show the state points, trajectories and $\Gamma$ for $N=N'=10$ in Figure \ref{Fig_case4_domain}, where the shaded region enclosed by $\Gamma$ and outside trajectories is exactly $\Omega$.

\subsubsection{Convergence Rate with Respect to $h$}
Since the loss function of the discovery on a compact region is merely the sum of loss functions of the discovery on every involved trajectory (see \eqref{29_2}), the implementation for the discovery on a compact region should share the same properties as the implementation on a trajectory, including the optimization errors and implicit regularization. The tests with respect to these properties will not be repeated in this example. Instead, we perform the test of the convergence rate with respect to $h$ to valid the error estimate that $e_{\hf}=O(h^p)$ if the network size is large enough. We take A-B $(M=1,\cdots,4)$ and BDF $(M=1,\cdots,4)$ schemes for $h=0.1,0.05,\cdots,0.1/2^4$, then compute the training and testing errors (shown in Figure \ref{Fig_case4_errors_h_method}). The theoretical orders of error decay are observed when $h$ is relatively large. While the overall error stops decreasing when $h$ is too small due to the dominance of the optimization error. Specifically, the 2-D profiles of the obtained approximate networks $\hf_j$ and the errors $\hf_j-f_j$ for $j=1,2$ are presented in Figure \ref{Fig_case4_profile}. The errors are observed to be below $O(10^{-3})$ everywhere in $\Omega$.

\begin{figure}
\centering
\subfloat[$e_{\hf}$ v.s. $h$]{
\includegraphics[scale=0.3]{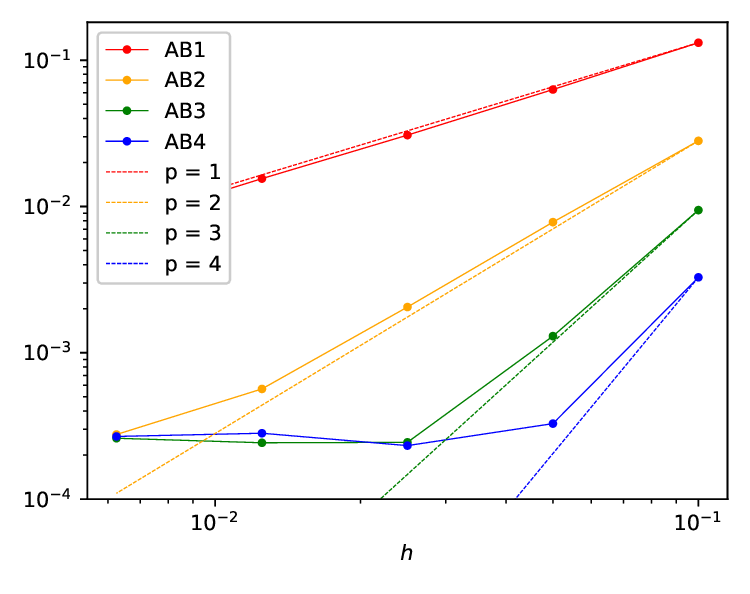}
\includegraphics[scale=0.3]{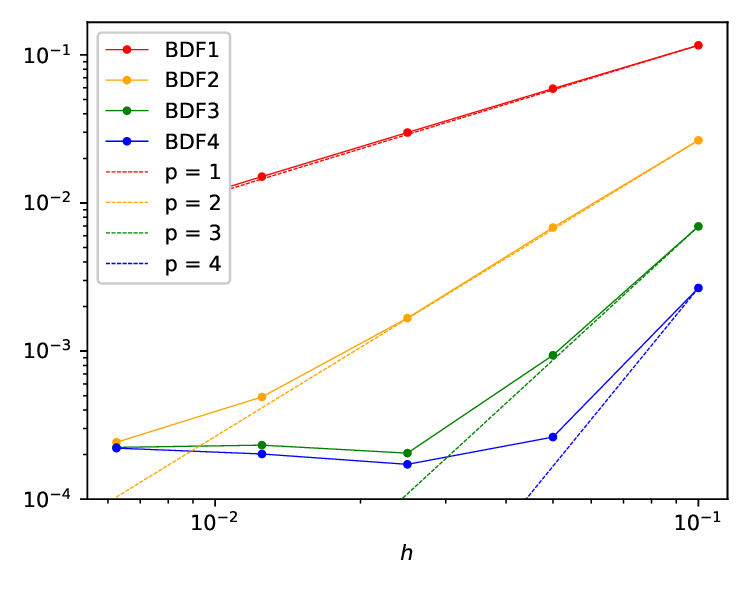}}
\subfloat[$\tilde{e}_{\hf}$ v.s. $h$]{
\includegraphics[scale=0.3]{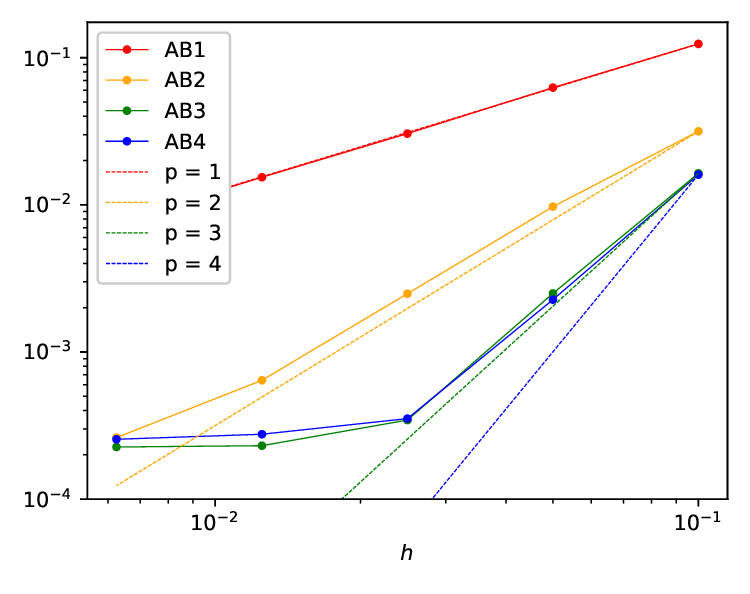}
\includegraphics[scale=0.3]{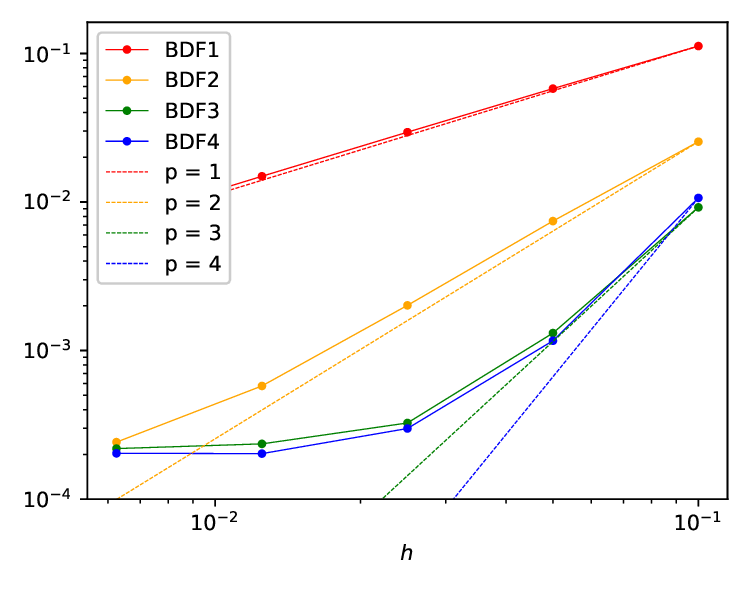}}\\
\caption{\em Training error $e_{\hf}$ and testing error $\tilde{e}_{\hf}$ versus $h$ via network-based A-B/BDF schemes of the model system \eqref{case4}}
\label{Fig_case4_errors_h_method}
\end{figure}

\begin{figure}
\centering
\subfloat[The first component]{
\includegraphics[scale=0.5]{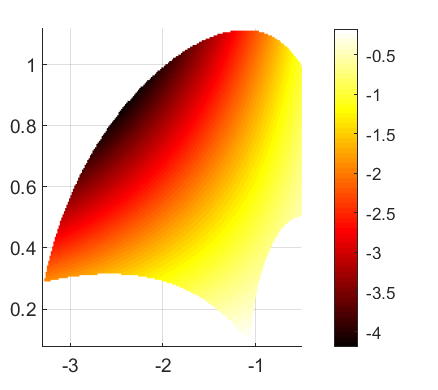}
\includegraphics[scale=0.5]{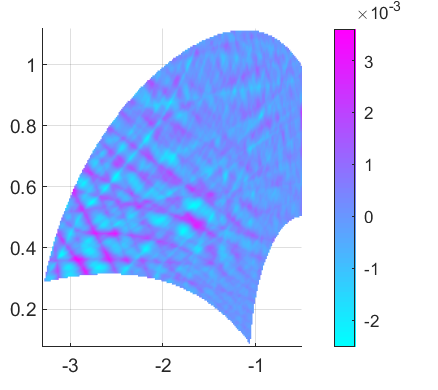}}
\subfloat[The second component]{
\includegraphics[scale=0.5]{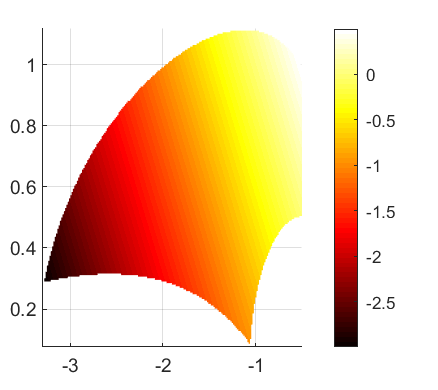}
\includegraphics[scale=0.5]{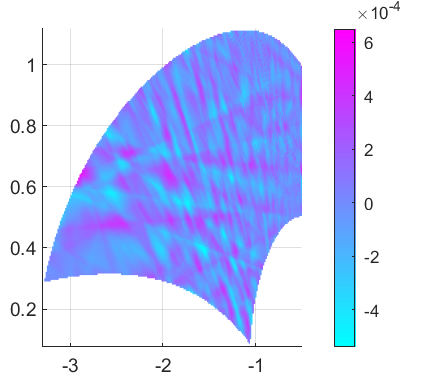}}\\
\caption{\em The profiles of obtained networks and errors of the model system \eqref{case4}}
\label{Fig_case4_profile}
\end{figure}

\section{Conclusion}
This paper presents a rigorous convergence analysis of the network-based LMMs that discover unknown dynamical systems. The main result shows that the $\ell^2$ grid error of the approximate function is bounded by $O(\kappa_2(\BA_h)(h^p+e_{\mathcal{A}}))$, where $\kappa_2(\BA_h)$ is the $2$-condition number of the corresponding matrix derived from the LMM scheme and $e_{\mathcal{A}}$ is the approximation error of the admissible set. This result is combined with approximation properties of deep neural networks to develop the error estimates for network-based LMMs. We also characterize the root condition to determine the uniform boundedness of $\kappa_2(\BA_h)$. Besides, several numerical experiments are conducted to validate our theory. We observe that the error decaying orders of various LMMs are close to the theoretical ones.

In the experiments, we also test the network-based method either using formulations without auxiliary conditions or using unstable LMM schemes. In theory, we can not guarantee the uniqueness of the solution at grid points in the former case, and we do not have upper bounds for the discovery error in the latter case. However, in practice, deep learning with gradient descent can still find solutions with errors in the similar ranges of their stable counterparts. More traditional approximations, such as grid functions and polynomials, are less robust and sensitive to the choice of solvers in comparison (see Appendix \ref{Sec_appendix}).

One limitation of our work is that the error estimation only quantifies the grid error, which is evaluated at the given sample locations. The generalization error out of sample locations is still theoretically unknown, though we observe excellent generalization performance in numerical experiments. Inspired by the works on generalization performance of deep learning for regression problems \cite{Kawaguchi2017,Mei2018,Mei2019}, decision problems \cite{Reppen2020} and PDEs \cite{Luo2020}, it is interesting to improve the error estimation from sample grid points to the whole trajectory. For example, the overlearning performance is studied in \cite{Reppen2020} using Rademacher complexity. Moreover, recurrent neural networks (RNNs) have been widely employed to build machine learning models of temporal data. The research on RNN generalization \cite{Allen-Zhu2019,Allen-Zhu2019_2,Oymak2019,Li2020} may shed light on the convergence analysis of the dynamics discovery.

Furthermore, our error analysis concentrates on the formulation with auxiliary conditions, while numerical tests show that the deep learning approach without auxiliary conditions can still perform well when the time step size is small enough. This might be due to the implicit regularization of the gradient descent and neural networks. Consequently, further investigation of the implicit regularization without auxiliary conditions is very interesting.

{\bf Acknowledgments.} Q. D. is supported in part by the US NSF CCF-1704833 and DMS-2012562. Y. G. is supported by Singapore MOE AcRF Grants R-146-000-271-112. C. Z. is supported by Singapore MOE (Ministry of Educations) AcRF Grants R-146- 000-271-112 and R-146-000-284-114 as well as NSFC Grant No. 11871364. H. Y. was partially supported by the US National Science Foundation under award DMS-1945029.

\bibliography{expbib}
\bibliographystyle{plain}

\appendix
\section{Supplementary Results on Unstable LMMs}\label{Sec_appendix}
Recall that $\kappa_2(\BA_h)$ denotes the 2-condition number of the matrix $\BA_h$ corresponding to certain LMM schemes. It has been shown in Theorem \ref{Thm03} and \cite{Keller2019} that as $N\rightarrow\infty$, $\kappa_2(\BA_h)$ is uniformly bounded for stable schemes. Similar arguments also show that $\kappa_2(\BA_h)$ increases linearly for marginally stable schemes and increases exponentially for unstable schemes. Although there has been no convergence theory for unstable schemes, it is intriguing to investigate how they perform in practice.

We first consider the discovery via linear system \eqref{21}, in which the target function is approximated by grid functions. Note that $\BA_h$ is a Toeplitz-type band matrix, and hence \eqref{21} are linear difference equations. For unstable schemes, the characteristic polynomial has roots of modulus greater than 1, which causes small perturbations of the system to grow exponentially in the solution. Specifically, let us consider the perturbed system of \eqref{21},
$\BA_h(\vec{\Bf}_h+\vec{\Bepsilon})=\left[\begin{array}{c}\Bc_h+\Bdelta\\\vec{\Bq}_h\end{array}\right]$, where $\Bdelta$ is a small perturbation of the initial value $\Bc_h$, and $\vec{\Bepsilon}$ is the error between the perturbed and original solutions. Then each component of $\vec{\Bepsilon}=[\varepsilon_s,\cdots,\varepsilon_{e(N)}]^T$ is given by $\varepsilon_n=c_1\lambda_1^n+c_2\lambda_2^n+\cdots+c_{N_a}\lambda_{N_a}^n$ for $n=s,s+1,\cdots,e(N)$, where $\lambda_1,\cdots,\lambda_{N_a}$ are the roots of the polynomial \eqref{39}, and $c_1,\cdots,c_{N_a}$ are completely determined by $\Bdelta$. For unstable schemes, at least one root $\lambda$ has modulus greater than 1, and hence the error component $\varepsilon_n$ grows exponentially as $n$ increases.

In practice, since $\BA_h$ is lower-triangular, it is natural to solve \eqref{21} by forward substitution directly. However, the error accumulation discussed above occurs in the process of forward substitution. To demonstrate this, we solve the linear system \eqref{21} concerning the unstable A-M scheme ($M=2$) to discover the dynamical system \eqref{case1}. We first use forward substitution and compute the relative discovery error $\left\|\vec{\Bf}_h'-\vec{\Bf}\right\|_2/\left\|\vec{\Bf}\right\|_2$, where $\vec{\Bf}_h'$ is the computed solution of the linear system and $\vec{\Bf}$ defined by \eqref{41} is the true governing function evaluated at grid points. It shows in Figure \ref{Fig_case1_AM2_LMM} that the discovery error increases rapidly as $h$ decreases, implying the failure of forward substitution.

We then repeat the test by employing iterative solvers such as the generalized minimal residual method (GMRES) with stopping residual $\tau=10^{-4}$. It shows in Figure \ref{Fig_case1_AM2_LMM} GMRES with this setting succeeds in obtaining decaying errors as $h$ decreases, whose orders are close to the theoretical ones \cite{Keller2019}. However, if we set a smaller stopping residual $\tau=10^{-8}$, GMRES also fails like the forward substitution. Similar results are observed when using biconjugate gradient method to solve the linear system. These comparative tests imply that the difficulty bought by unstable schemes can be lessened by using iterative solvers, but these solvers are still sensitive to the implementation parameters because of the ill-conditioning of the method.

\begin{figure}
\centering
\subfloat[FS]{
\includegraphics[scale=0.5]{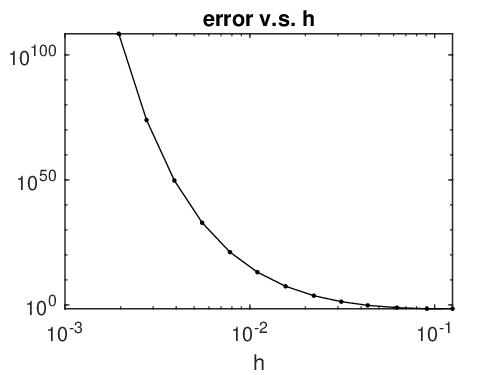}}
\subfloat[GMRES ($\tau=10^{-4}$)]{
\includegraphics[scale=0.5]{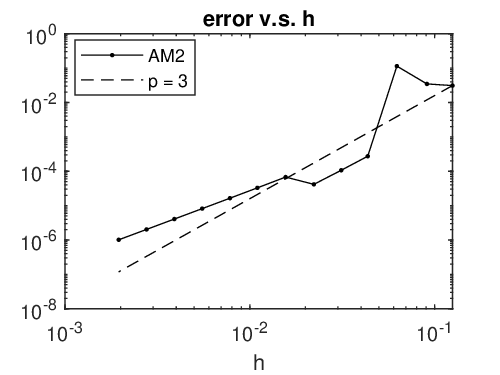}}
\subfloat[GMRES ($\tau=10^{-8}$)]{
\includegraphics[scale=0.5]{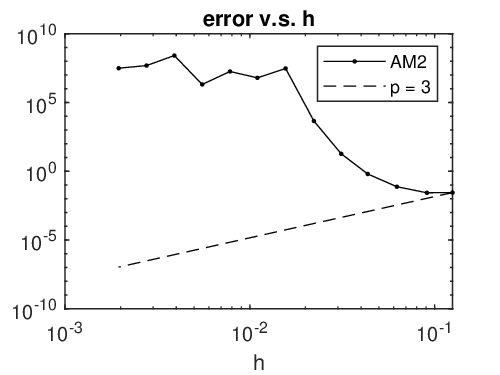}}
\caption{\em Discovery error versus $h$ using forward substitution (FS), GMRES ($\tau=10^{-4}$) or GMRES ($\tau=10^{-8}$) in the discovery of the model problem \eqref{case1}.}
\label{Fig_case1_AM2_LMM}
\end{figure}

Next, we consider the discovery using linear approximation forms. Suppose the approximation set $\mathcal{A}$ is a linear space with basis $\{\psi_1,\cdots,\psi_{d'}\}$, then the governing function can be approximated by the form $\hf_{\mathcal{A}}=c_1\psi_1+\cdots+c_{d'}\psi_{d'}$ with coefficients $c_1,\cdots,c_{d'}$ to be determined. Note that $\mathcal{A}$ can be spaces of polynomials, finite elements, splines, etc. Under the LMM framework, we aim to compute $c_1,\cdots,c_{d'}$ such that
\begin{equation}\label{42}
\BA_h\BPsi_h\vec{\Bc}=\left[\begin{array}{c}\Bc_h\\\vec{\Bq}_h\end{array}\right],
\end{equation}
where $\BPsi_h:=[\psi_i(\Bx_n)]_{n=s,\cdots,e(N)}^{i=1,\cdots,d'}$ and $\vec{\Bc}:=[c_1,\cdots,c_{d'}]^T$. Note that \eqref{42} is a linear system similar to \eqref{21} but might be square if $d'=t(N)$, overdetermined if $d'<t(N)$ or underdetermined if $d'>t(N)$. It is natural to solve \eqref{42} by first solving
\begin{equation}\label{43}
\BA_h\vec{\By}=\left[\begin{array}{c}\Bc_h\\\vec{\Bq}_h\end{array}\right]
\end{equation}
for $\vec{\By}$, then solve $\BPsi_h\vec{\Bc}=\vec{\By}$ for $\vec{\Bc}$. However, solving \eqref{43} faces the same issue as the linear system \eqref{21} discussed above.

Therefore, it implies that with unstable LMM schemes, both grid function approximation and linear form approximation are less robust due to the ill-conditioning. One might attempt to overcome such difficulties by developing effective preconditioners for the linear system \eqref{21} or \eqref{42}, at least when there is no high demand on the numerical precision.

In comparison, the network approximation shows more robustness in practice to get solutions within the ranges of optimization errors (Section \ref{Sec_case1_AM}), which is conjectured to be a consequence of the implicit regularization. All these attempts and conjectures may be further studied in future work.

\end{document}